\documentclass[11pt,reqno]{amsart}

\usepackage{amsfonts,color}
\usepackage{graphics,graphicx}
\usepackage{indentfirst}
\usepackage{cite}
\usepackage{latexsym,hyperref}
\usepackage{amsmath}
\usepackage{amssymb}
\usepackage[dvips]{epsfig}
\usepackage{amscd}

\allowdisplaybreaks

\newtheorem{theorem}{Theorem}[section]
\newtheorem{lemma}{Lemma}[section]
\newtheorem{corollary}{Corollary}[section]
\newtheorem{proposition}{Proposition}[section]

\theoremstyle{definition}
\newtheorem{definition}{Definition}[section]

\theoremstyle{remark}
\newtheorem{remark}{Remark}[section] 

\newcommand{\n}{\rho}
\newcommand{\ti}{\tilde}
\renewcommand{\div}{ {\rm div}}
\newcommand{\na}{\nabla }

\newcommand{\ga}{\gamma}
\newcommand{\de}{\delta}

\newcommand{\la}{\label}

\newcommand{\bnn}{\begin{eqnarray*}}
\newcommand{\enn}{\end{eqnarray*}}

\newcommand{\ba}{\begin{aligned}}
\newcommand{\ea}{\end{aligned}}
\newcommand{\be}{\begin{equation}}
\newcommand{\ee}{\end{equation}}

\def\p{\partial}

\def\lap{\triangle}

\def\ep{\varepsilon}
\def\a{\alpha}
\def\om{\Omega}

\def\R{\mathbb{R}}

\makeatletter
\@addtoreset{equation}{section}
\makeatother
\title{Energy equality in compressible fluids with physical boundaries}
\author[R. M. Chen]{Robin Ming Chen}
\address{Department of Mathematics, University of Pittsburgh, Pittsburgh, PA 15260} 
\email{mingchen@pitt.edu}
\author[Z. L. Liang]{Zhilei Liang}
\address{School of Economic Mathematics, Southwestern  University of Finance and Economics, Chengdu  611130,  China} 
\email{zhilei0592@gmail.com}
\author[D. Wang]{Dehua Wang}
\address{Department of Mathematics, University of Pittsburgh, Pittsburgh, PA 15260} 
\email{dwang@math.pitt.edu}
\author[R. Xu]{Runzhang Xu}
\address{College of Science, Harbin Engineering University, Harbin 150001, P. R. China} 
\email{xurunzh@hrbeu.edu.cn}

\date{\today}

\begin{document}

\maketitle   

 \begin{abstract} 
We study the energy balance for the weak solutions of the three-dimensional compressible Navier--Stokes equations in a bounded domain. We  establish an $L^{p}$-$L^{q}$ regularity condition  on the velocity field for the energy equality to hold, provided that the density is bounded and satisfies $\sqrt{\n} \in L^\infty_t H^1_x$.    
The main idea is to construct a global mollification combined with an independent boundary cut-off, and then take a double limit to prove the convergence of the resolved energy.
\end{abstract}


 \section{Introduction}
 
In fluid mechanics, compressible fluids play an important role in many fields of applications, including astrophysics (star-formation, interstellar/intergalactic medium), engineering (supersonic aircraft, gas turbines, combustion engines), and so on. 
In this paper we consider the following three-dimensional Navier--Stokes equations of isentropic compressible flows, consisting of the conservation of mass and momentum,
\begin{equation}\label{1.1}
\left\{\ba
& \n_{t}+{\rm div} (\n u)=0,\\
& (\n u)_t +{\rm div}(\n u\otimes u )+ \na P(\rho) 
-\mu\lap u-(\mu+\lambda)\na \div u=0, \ea \right.
\end{equation}
where $\rho \ge 0$ is the   density of the flow, $u\in\R^3$ is the velocity, and $P(\n) = \rho^\ga$ is the pressure 
with    $\ga > 1$ . The viscosity constants include the shear viscosity $\mu > 0$ and the bulk viscosity $\lambda$ satisfying  $\lambda + {2\over 3}\mu \ge 0$. 

We are particularly interested in the behavior of the compressible flows confined within solid walls. Such flows are ubiquitous in nature as well as in applications.     Mathematically we consider the above system \eqref{1.1} in an open bounded domain $\om \subset \mathbb{R}^3$ and pose the usual no-slip boundary condition
\be\la{1a}u=0\,\,\,\,{\rm on}\,\,\,\p\om.\ee 
Finally we complement \eqref{1.1} with the initial condition
\be\la{1b} \n(x,0)=\n_{0}(x),\quad (\n u)(x,0)=(\n_{0}u_{0})(x),\quad x\in \om,\ee
where we define $u_{0}=0$ on the sets $\{x\in \om\,\,:\,\,\n_{0}=0\}.$

System \eqref{1.1}--\eqref{1b} possesses an energy balance law that  holds at least formally for strong solutions:
\be\la{energy} 
\begin{split}
\int_{\om}\left(\frac{1}{2}\n|u|^{2}+\frac{\n^{\ga}}{\ga-1}\right) dx\  + & \int_{0}^{t}\int_{\om} \left( \mu|\na u|^{2}+(\mu+\lambda)(\div u)^{2} \right)\,dxds \\
& = \int_{\om}\left(\frac{1}{2}\n_0|u_0|^{2}+\frac{\n_0^{\ga}}{\ga-1}\right)dx.
\end{split}
\ee
On the other hand, from the classical results of Lions \cite{lion2} and Feireisl \cite{Feireisl1,fei}, this system also allows for solutions with less regularity, namely the weak solutions (see below), which only satisfy an energy inequality. 
\begin{definition} \la{defi}  
For a given $T>0$, we call $(\n,u)$ a weak solution on $[0, T]$ to \eqref{1.1}--\eqref{1b} if 
\begin{itemize}
\item The problem \eqref{1.1}-\eqref{1b}  holds  in $\mathcal{D}'([0,T)\times \om)$ and
\begin{equation}\label{1.12}
 \n^{\ga},\,\,\n |u|^{2}  \in L^{\infty}\left(0,T;  L^{1}(\om)\right),\quad   u \in L^{2}\left(0,T;H_{0}^{1}(\om) \right).
\end{equation}   
\item $(\n,u)$ is a renormalized solutions of $\eqref{1.1}_{1}$ in the sense of  \cite{di}.  
\item The energy inequality holds  
\be\la{1.8} 
\begin{split}
\int_{\om}\left(\frac{1}{2}\n|u|^{2}+ \frac{\n^{\ga}}{\ga-1}\right) dx & +\int_{0}^{t}\int_{\om}\left( \mu|\na u|^{2}+(\mu+\lambda)(\div u)^{2} \right) dxds \\
& \le  \int_{\om}\left(\frac{1}{2}\n_{0}|u_{0}|^{2}+\frac{\n_{0}^{\ga}}{\ga-1}\right)dx.
\end{split}
\ee
\end{itemize}
\end{definition}

The lack of the exact equality in \eqref{1.8} is reminiscent of the energy inequality of the Leray--Hopf solution to the incompressible Navier--Stokes equations, which still remains open up to date. This question is very well motivated from physical grounds. 
 Validity of  the energy equality \eqref{energy} would rule out the possibility of interior anomalous energy dissipation (i.e., the energy dissipation does not vanish as the viscosity goes to zero), an effect experimentally observed or numerically evidenced in turbulent flows \cite{sr,ka,pe}.   This  is also associated with weak solutions of the Euler equations in the framework of the celebrated Onsager conjecture \cite{on}.

One of the main difficulties in establishing the energy equality in the absence of the boundary (i.e. $\om = \mathbb{R}^3$ or $\mathbb{T}^3$) lies in the fact that the regularized velocity field and density may generate a non-vanishing energy flux due to the nonlinear coupling. For incompressible flows with constant density, J. L. Lions \cite{lions} proved that energy equality holds for $u\in L^4_{t,x}$. This was reproduced by Lady$\breve{\textrm{z}}$enskaja et al. \cite{lady} in the general context of parabolic equations. In \cite{se} Serrin gave a dimension-dependent condition $u\in L^p_tL^q_x$ for ${2\over p} + {n\over q} \le 1, \ q > n$,  where $n$ is the space dimension. 
Later Shinbrot in \cite{shi} removed the dimensional dependence and improved the conditions to  ${2\over p} + {2\over q} \le 1, \ q \ge 4$.
An alternative proof of Shinbrot's result can be found in \cite{yu3}. New types of conditions have been obtained recently, including Besov-type regularity conditions \cite{che,de}, weak-in-time with optimal Onsager spatial regularity conditions \cite{che2}, new $L^p_tL^q_x$ conditions in combination with low dimensionality of the singular set \cite{shv2}, to name a few. For inhomogeneous incompressible flows Leslie-Shvydokoy \cite{shv}  proved the energy equality in Besov spaces.  Concerning compressible fluids, the theoretical study is more recent. In \cite{de1,de2} Drivas-Eyink followed the approaches of \cite{ey,con,du} in the framework of Onsager's theory to derive necessary conditions for dissipative anomalies
of kinetic energy in turbulent solutions of the compressible Euler equations. Feireisl et al. \cite{fei1} gave sufficient Besov regularity conditions on the weak solutions for energy conservation of the compressible Euler system, excluding the case of vacuum. Regularity conditions for energy conservation which allow the presence of vacuum in the compressible Euler flow were provided  by Akramov et al. \cite{ADSW}.  For the energy equality of the compressible Navier--Stokes \eqref{1.1}, Yu \cite{yu1} proved that \eqref{energy} holds true if the velocity variable $u$ satisfies an $L^p_tL^q_x$ condition while the density $\rho$ is bounded 
and  $\sqrt{\n}\in L^{\infty}(0,T; H^{1})$.
%

The presence of solid boundaries makes the dissipative mechanism more complex. Vortical structures are organized in the viscous boundary layers that detach from the walls. The boundary layer becomes thinner as the viscosity decreases and generates sharp velocity gradients which can propagate into the bulk of the fluid, making the interior velocity field irregular to sustain anomalous energy dissipation. Therefore mathematically, the added challenge comes from controlling the regularity of the solutions near the boundary in order to pass from local to global energy balance. The first result addressing the Onsager's theory for wall-bounded flows is due to Bardos--Titi \cite{bar} in the context of the incompressible Euler equations under the assumption of a global regularity on the velocity. Such a result was further refined by  Bardos--Titi--Wiedemann \cite{bar1} and Drivas--Nguyen \cite{ngu} 
where a weaker assumption is used that is consistent with the formation of the boundary layer in the vanishing viscosity limit. In line with the method of \cite{bar1}, Akramov et al. \cite{ADSW} were able to treat the case of compressible Euler flows confined in a bounded domain.  In \cite{BGGTW}, Bardos et al. managed to extend and prove the Onsager conjecture for a class of conservation laws that admit a generalized entropy. 
The idea of \cite{bar} was also exploited by Yu \cite{yu2} for the incompressible Navier--Stokes equations in a bounded domain, obtaining the same Shinbrot type interior regularity criterion, with an additional Besov regularity on the velocity to handle the boundary effects coming from the diffusion term. 

The basic strategy used in \cite{ADSW,bar,bar1,ngu,yu2} is localization. Specifically, an additional cut-off function was introduced that separates the boundary part from the interior domain. The distance $h$ from the support of the boundary cut-off to the boundary is chosen to be large enough compared with the scale $\ep$ for the mollification, leaving {enough space} to mollify the interior velocity. This way the interior regularity criterion can be achieved following the classical commutator estimates in the spirit of \cite{ey,con,du}. To obtain the global energy balance, one needs to patch the interior estimates with the estimates on the boundary layer. This is done by carefully examining the scale-transfer terms in the bulk and at the boundary. 
For the incompressible Euler equations, to ensure that the inertial boundary production vanishes in the double limit $\ep, h \to 0$, one needs to assume continuity of the normal component of the energy flux near the boundary \cite{bar1}, which is essentially equivalent to assuming continuity of the near-wall normal velocity \cite{ngu} due to the non-penetration boundary condition. The case for the incompressible Navier--Stokes is slightly more delicate. The boundary production includes an additional contribution coming from the diffusion term, which involves the information about the velocity gradient near the boundary. However, such information cannot be inferred from the no-slip boundary condition, and this is the reason why in \cite{yu2} an extra Besov regularity on the velocity is assumed. 

\subsection{Methodology}

The goal of this paper is to  understand the relation between the energy equality and the regularity of the solutions in the appearance of the boundary.
For this purpose we shall  introduce a  new  approach  different from \cite{bar} and apply it to the compressible Navier--Stokes system \eqref{1.1}--\eqref{1b}. 
As a result of our new approach, we can avoid assuming additional regularity of the velocity near the boundary as in \cite{yu2} in order to deal with the boundary production due to the diffusion terms.
To the best of our knowledge, our paper appears to be the first work addressing the energy balance of flows both in the compressible regime and in a bounded domain.

In the paragraphs below we briefly describe the ideas of our method.

\medskip

\noindent{\it Global mollification.} The approach we propose in this paper is ``global" in the sense that we do not shrink the domain $\om$ to create space for the mollification. Instead, the mollified functions are defined globally in the whole domain $\om$. Roughly speaking, the interior mollification will be the same as in the general localization approach. However, for the boundary part, when $\partial\om$ is reasonably smooth, we introduce a local variable shift toward the interior of $\om$ and then perform the usual mollification. Finally we obtain a global approximation by gluing together the boundary and interior parts using a partition of unity. The details are given in Section \ref{subsec_global}. We want to point out that such an approximation is in the spirit of the one discussed in \cite[Section 5.3]{evans}.

The regularization of the momentum equation in \eqref{1.1} can be done the same way: performing  local mollifications, and then summing them up according to the partition of unity.

\medskip

\noindent{\it Test functions.} The global approximation avoids cutting out the boundary information, at the price that the mollified velocity field fails to vanish on the boundary. Therefore one still needs to introduce a boundary cut-off function supported $\delta$-distance away the boundary (cf. \eqref{phi}), and multiply it to the mollified velocity to construct the test function. The difference, compared with \cite{bar,bar1,ngu,yu2}, is that the mollification scale $\ep$ and boundary cut-off scale $\delta$ are completely independent. This leaves much freedom for the choice of $\delta$ and could be useful for other applications, for instance, the study of anomalous dissipation in the vanishing viscosity limit, which will be addressed in a forthcoming paper \cite{cl}.

\medskip

\noindent{\it Boundary production due to diffusion.} Similarly to \cite{yu2}, the inertial boundary production includes terms that involve the gradient of the velocity field $\nabla u$ which comes from the diffusion terms. As explained earlier, it is hard to control such terms directly due to the lack of boundary condition on $\nabla u$. Here we will first pass the limit as $\ep \to 0$, leaving $\delta$ fixed, so that we recover the full velocity in the resulting approximated energy equality \eqref{m3}. This allows us to employ the classical Hardy type inequality (cf. Lemma \ref{lemma}) to annihilate the boundary contribution from the diffusion terms. Note that the crucial ingredient in this argument is the fact that $\ep$ and $\delta$ are independent.

\medskip

\noindent{\it Commutator estimates.} In proving energy conservation/equality, the
commutator estimates are required for treating the nonlinear terms. Compared with incompressible homogeneous equations, a notable difference in compressible (or inhomogeneous) equations is that the momentum equation contains a time derivative of a nonlinear term $\n u$, and hence it needs a commutator estimate in time. We follow the ideas in  \cite{yu1} 
in order to allow for vacuum states, with slight modifications to work in the Sobolev spaces; see Corollary \ref{co}.   

\subsection{Main results}
Our energy equality criterion for the compressible Navier--Stokes equations \eqref{1.1}--\eqref{1b} is
\begin{theorem}\la{t1} 
Let  $\om$ be an  open, bounded domain with   $C^{1}$ boundary $\p \om$, and   $(\n, u)$ be   a weak solution in  Definition \ref{defi}.  Assume that  
\be\la{1.5a}
\ba& 0\le \n\le \bar{\n}<\infty,\quad \na \sqrt{\n}\in L^{\infty}\left(0,T;L^{2}(\om)\right).\ea
\ee  
 If \be\la{1.5}\ba&   u\in L^{p}(0,T;L^{q}(\om)),\quad p\ge4,\,q\ge6,\ea\ee
 and moreover, 
 \be\la{1.5s} u_{0}\in L^{q_{0}}(\om),\quad q_{0}>3.\ee  
Then  \eqref{energy}  holds for any $t\in[0, T]$.  
\end{theorem}

A few remarks are in order as follows.
\begin{remark}  Condition  \eqref{1.5} can be improved in the absence of the vacuum states to 
  \be\la{1.5r}\ba&   u\in L^{p}(0,T;L^{q}(\om)), \quad {\rm with}\quad \frac{2}{p}+\frac{3}{2q}\le \frac{3}{4},\quad q\ge6.\ea\ee
 In fact,  it follows from  \eqref{1.12}  and \eqref{1.5a}  that   $u\in L^{\infty}(0,T;L^{2}(\om)),$  which implies  by interpolation  
\bnn\ba \|u\|_{L^{4}(0,T;L^{6})}\le \|u\|_{L^{\infty}(0,T;L^{2})}^{\frac{(q-6)}{3(q-2)}} \|u\|_{L^{\frac{8q}{3(q-2)}}(0,T;L^{q})}^{\frac{2q}{3(q-2)}}\le C\|u\|_{L^{p}(0,T;L^{q})},\ea\enn
as long as $ \frac{8q}{3(q-2)}\le p.$
This is  condition  \eqref{1.5r}.   
\end{remark} 

\begin{remark} We will apply the same idea to treat the Leray--Hopf solution of the incompressible Navier--Stokes equations in a bounded domain in the Appendix. We are able to obtain an analogous regularity criterion as in the periodic case, with an additional condition on the control of the pressure on the boundary. This removes the extra Besov regularity assumption on the velocity as in \cite{yu2}.
\end{remark} 
  
 
\begin{remark}  The regularity assumption  \eqref{1.5a} on the density  is critical for making commutator estimates work, but  it is not optimal.    Alternatively,  it can be relaxed at the expense of  imposing extra time regularity on  velocity field.  This is similar to, for e.g., \cite{de2,fei1}.   
\end{remark}

The  rest  of this paper  is  organized as follows:   In Section \ref{sec_prep} we construct the global mollification, prove the commutator estimates in Sobolev spaces, and recall a classical Hardy-type inequality. In Section \ref{sec_proof} we give the proof for the main theorem. Finally in the Appendix we apply our method to the incompressible Navier--Stokes equations in a bounded domain and give sufficient regularity conditions for the energy equality. 
  \bigskip
  
\section{Preliminaries}\label{sec_prep}
 
 \subsection {Global approximation in   $\om$}\label{subsec_global}
If   $f\in L^{p}(0,T;W^{1,p}(\om))$,  the following  local    approximation is well known 
 \be\la{ws16}  f^{\ep}\rightarrow f\quad{\rm in}\quad  L^{p}_{loc}(0,T;W_{loc}^{1,p}(\om)),\quad\forall\,\,\, p\in [1,\infty),\ee
where  
 \be\la{ws16a} f^{\ep}(x,t)=\int_{0}^{T}\int_{\om}f(y,s)\eta_{\ep}(x-y,t-s)dyds,\quad\eta_{\ep}(x,t)=\frac{1}{\ep^{4}}\eta\left(\frac{x}{\ep},\frac{t}{\ep}\right),\ee 
with $\eta(x,t)$ being   the    standard    mollifier supported in a unit ball.

For the purpose of this paper,   we adopt some  ideas  in \cite[Section 5.3]{evans} and  build a global  approximation in  $L^{p}_{loc}(0,T;W^{1,p}(\om))$.

$(1)$  Since $\p\om\in C^{1},$   for a fixed  $x_{1}\in \p\om,$ there exist some $r_{1}>0$ and a $C^{1}$ function $h\,:\mathbb{R}^{2} \to \mathbb{R}$ such that, upon relabelling  the coordinate axis  if necessary,  we have
\bnn \om\cap B(x_{1},r_{1})=\{x\in B(x_{1},r_{1})\,:\,\,x_{3}>h(x_{1},x_{2})\},\enn where $B(x,r)$ is an open  ball which  centers  in $x$ with radius $r.$

Let  $  V_{1}=\om\cap B(x_{1},\frac{r_{1}}{2}).$ For a small $\ep<\frac{r_{1}}{8},$ we define the shifted point  
\begin{equation}\la{shifted}
x^\varepsilon : = x - \ep \vec{n}(x_1),
\end{equation}
then it is obvious   that   
\bnn B(x^\ep,\,\ep) \subset \om\cap B(x_{1},r_{1}),\quad \text{for all } \ x\in V_{1},\enn
 where $\vec{n}(x_1) $ is the unit outward normal vector of $\partial \om$ at $x_1$ (see Fig. \ref{globalapprox} below).  
\begin{figure}[h]
  \includegraphics[scale=1]{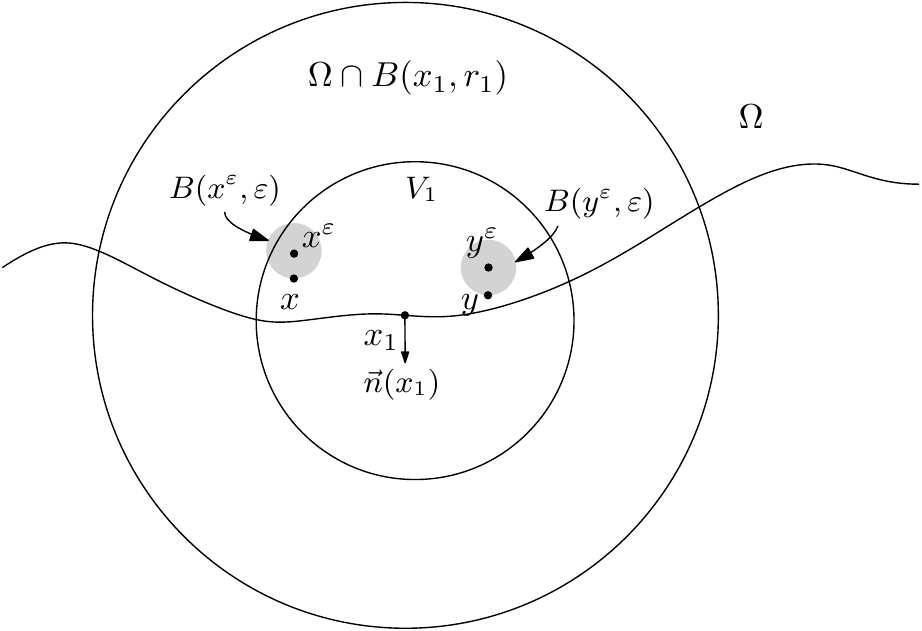}
  \caption{The local variable shift defined in $V_1$.}
  \label{globalapprox}
\end{figure}
 
 Define the shifted function 
 \be\la{li1} \ti{f}(x,t)=f(x^\ep,t),\quad x\in V_{1}.\ee 
 Then   there is room to   mollify $\ti{f}(x,t)$ like   \eqref{ws16a},     that is,   \be\la{ws3}\begin{split}   \ti{f}^{\ep}_{1}(x,t) & = \int_{0}^{T}\int_{\mathcal{V}_{1}}\ti{f}(y,s)\eta_{\ep}(x-y,t-s)dyds \\
 & = \int_{0}^{T}\int_{\mathcal{V}_{1}-\ep \vec{n}(x_1)} f(y,s)\eta_{\ep}(x^\ep-y,t-s)dyds,\end{split}\ee  for every $ (x,t)\in  V_{1}\times (\ep,T-\ep),$
and $\mathcal{V}_1$ can  be simply taken to be $B(x_1, \frac{r_1}{2} + 2\ep) \cap \om$.
  
We   claim that
\be\la{ws10} \lim_{\ep\rightarrow0}\|\ti{f}^{\ep}_{1}-f\|_{ L^{p}_{loc}(0,T;W^{1,p}(V_{1}))}=0.\ee
To confirm  this,    for any multi-index $\a$ satisfying $|\a|\le 1,$  
\bnn    \| \p_{x}^{\a}(\ti{f}^{\ep}_{1}-f)\|_{L^{p}_{loc}([0,T]\times V_{1})} \le \|  \p_{x}^{\a}(\ti{f}^{\ep}_{1}-\ti{f})\|_{L^{p}_{loc}([0,T]\times V_{1})}+\| \p_{x}^{\a}(\ti{f}-f)\|_{L^{p}_{loc}([0,T]\times V_{1})}.\enn
The second term on the right-hand side of the above goes to zero as $\ep\rightarrow0$ because the  translation is continuous in $L^{p}$, and the first term also vanishes as $\ep$ goes to zero   due to \eqref{ws16}.  
 
$(2)$ Since $\partial \om$ is compact, we find finitely many points $x_{i}\in \p\om$, radii $r_{i}>0$, corresponding sets $V_{i} = \om \cap B(x_i, \frac{r_i}{2})$ and functions $\ti{f}^{\ep}_{i}\in C^{\infty}\left(\overline{V_{i}}\right)$ ($i = 1,\ldots,k$) 
 such that   $\p \om \subset \bigcup_{i=1}^{k} \overline{V_{i}}$.   By  \eqref{ws10},    it follows that  
for a  given  $\de>0$, \be\la{ws17} \|\ti{f}^{\ep}_{i}-\ f\|_{L^{p}_{loc}(0,T;W^{1,p}(V_{i}))}<\frac{\de}{k+1}, \quad  \,\,\,i\in\{1,\cdot\cdot\cdot k\},\ee as  long as $\ep$ is taken  small.   
Take also an open set $V_{0}\subset\subset \om$ to satisfy 
 \be\la{ws18}\om\subset   \bigcup_{i=0}^{k}V_{i}\quad{\rm and}\quad \|f^{\ep}-\ f\|_{L^{p}_{loc}(0,T;W^{1,p}(V_{0}))}<\frac{\de}{k+1}.\ee 
   
We have the following   
\begin{proposition}[Global  mollification  in  $\om$] \la{po}
 Let $\{\xi_{i}\}_{i=0}^{k}$ be a smooth partition of unity subordinate to $V_{i}$,  that is,
\be\la{u2}\ba 0\le \xi_{i}\le 1,\quad   \xi_{i}\in C^{\infty}(V_{i}),\  {\rm{supp }}\ \xi_{i} \subset V_{i},\quad \sum_{i=0}^{k}\xi_{i}=1\,\,{\rm in}\,\, \om.\ea\ee
 Define 
\be\la{ws15} [f]^{\ep}(x,t) :=\xi_{0}(x)f^{\ep}(x,t)+\sum_{i=1}^{k}\xi_{i}(x)\ti{f}_{i}^{\ep}(x,t),\quad \forall\,\,x\in \om,\ee
where $f^{\ep}(x,t)$ and $\ti{f}_{i}^{\ep}(x,t)$  are defined in \eqref{ws16} and \eqref{ws3} respectively. 

Then, sending $\ep\rightarrow0,$  we have  \be\la{p9a}   [f]^{\ep}\rightarrow f  \quad {\rm in}\,\,\,L^{p}_{loc}(0,T;W^{1,p}(\om)).\ee
Moreover,  \be\la{p9} [f]^{\ep}-f^{\ep}\rightarrow 0\quad {\rm in}\,\,\,L^{p}_{loc}(0,T;W^{1,p}(V_{0}))\quad{\rm and }\quad [f]^{\ep}-\ti{f}_{i}^{\ep}\rightarrow 0\quad {\rm in}\,\,\,L^{p}_{loc}(0,T;W^{1,p}(V_{i})).\ee
\end{proposition}
\begin{proof}.  
It follows from  \eqref{ws17},  \eqref{ws18}, \eqref{u2} that for any multi-index $\a$ satisfying $|\a|\le 1,$ 
\bnn\ba &\|\p_{x}^{\a}([f]^{\ep} -f)\|_{L^{p}_{loc}(0,T;L^{p}(\om))} \\
&\le \|\p_{x}^{\a}(\xi_{0}f^{\ep}- \xi_{0}f)\|_{L^{p}_{loc}(0,T;L^{p}(V_{0}))}+ \sum_{i=1}^{k}\|\p_{x}^{\a}(\xi_{i}\ti{f}_{i}^{\ep} - \xi_{i}f)\|_{L^{p}_{loc}(0,T;L^{p}(V_{i}))}\\
&\le C \|f^{\ep}-f\|_{L^{p}_{loc}(0,T;W^{1,p}(V_{0}))}+C\sum_{i=1}^{k}\|\ti{f}_{i}^{\ep}-f\|_{L^{p}_{loc}(0,T;W^{1,p}(V_{i}))}\le C\de. \ea\enn    
  This proves \eqref{p9a}.  

As a  direct result of \eqref{p9a}, \eqref{ws16} and \eqref{ws10},  we obtain  \eqref{p9}.    
\end{proof}

\subsection{Commutator estimates}
\begin{lemma}[Lemma 2.3 in \cite{lion1}] \la{lem} Suppose  $\n\in W^{1,r_{1}}([0,T]\times \om),\,\,\, u\in L^{r_{2}}([0,T]\times \om)$, and  $1\le r,\, r_{1},\,r_{2}\le \infty,\,\,  \frac{1}{r_{1}}+\frac{1}{r_{2}}=\frac{1}{r}.$  Then, 
  
\be\la{p5}\|\p(\n u)^{\ep}-\p(\n u^{\ep})\|_{L^{r}_{loc}([0,T]\times \om)}\le \|u\|_{L^{r_{2}}([0,T]\times \om)}\|\p\n\|_{L^{r_{1}}([0,T]\times \om)},\ee
  where  $\p=\p_{t}$ or $\p=\p_{x}$, and $f^{\ep}$ is defined  as in  \eqref{ws16}.     Furthermore, 
   \be\la{6} \p(\n u)^{\ep}-\p(\n u^{\ep})\rightarrow 0\,\,\,{\rm in}\,\,\,L^{\underline{r}}_{loc}([0,T]\times \om),\quad {\rm as}\,\,\,\ep\rightarrow0,\ee
   where  $\underline{r}=r$ if $r_{2}<\infty$  and   $\underline{r}<r$ if $r_{2}=\infty.$ 
 \end{lemma}
  
We will  need   the following     variant of Lemma \ref{lem}.
  \begin{corollary} \la{co}   
Under  the same  assumptions listed  in Lemma \ref{lem}, we have  
 \be \la{p1} \|\p_{t}\left((\ti{\n}\ti{u})_{i}^{\ep}-\n \ti{u}_{i}^{\ep}\right)\|_{L^{r}_{loc}(0,T;L^{r}(V_{i}))}\le C\|u\|_{L^{r_{2}}([0,T]\times \om)}\left(\|\p_{t}\n\|_{L^{r_{1}}([0,T]\times \om)}+\|\na \n\|_{L^{r_{1}}([0,T]\times \om)}\right)\ee  
   and \be\la{p2}\|\p_{x}\left((\ti{\n}\ti{u})_{i}^{\ep}-\n \ti{u}_{i}^{\ep}\right)\|_{L^{r}_{loc}(0,T;L^{r}(V_{i}))}\le C\|u\|_{L^{r_{2}}([0,T]\times \om)}\|\na \n\|_{L^{r_{1}}([0,T]\times \om)},\ee where $V_{i}\,(i=1,\cdot\cdot\cdot k)$ is the same as mentioned earlier,  and $\ti{f}_{i}^{\ep}$ is defined  in \eqref{ws3}. 
Furthermore,   
   \be\la{p3}\p\left((\ti{\n}\ti{u})_{i}^{\ep}-\n \ti{u}_{i}^{\ep}\right)\rightarrow 0\,\,\,{\rm in}\,\,\,L^{\underline{r}}_{loc}(0,T;L^{\underline{r}}(V_{i}))\quad {\rm as}\,\,\,\ep\rightarrow0,\ee
 where $\underline{r}$ is given as in Lemma \ref{lem}.
 \end{corollary}
\begin{proof} 
The proof is the same spirit of Lemma \ref{lem}.  By \eqref{p5},    
\be\la{b3}\ba \|&\p_{t}((\ti{\n} \ti{u})_{i}^{\ep}- \n \ti{u}_{i}^{\ep})\|_{L^{r}_{loc}(0,T;L^{r}(V_{i}))}\\
  &\le \|\p_{t}((\ti{\n} \ti{u})_{i}^{\ep}-\ti{\n} \ti{u}_{i}^{\ep})\|_{L^{r}_{loc}(0,T;L^{r}(V_{i}))}+\|\p_{t}(\ti{\n} -\n)\ti{u}_{i}^{\ep}\|_{L^{r}_{loc}(0,T;L^{r}(V_{i}))} \\
  & \quad +\|(\ti{\n} -\n)\p_{t}\ti{u}_{i}^{\ep}\|_{L^{r}_{loc}(0,T;L^{r}(V_{i}))}\\
  &\le C\|u\|_{L^{r_{2}}([0,T]\times \om)} \|\p_{t}\n\|_{L^{r_{1}}([0,T]\times \om)} +\|(\ti{\n} -\n)\p_{t}\ti{u}_{i}^{\ep}\|_{L^{r}_{loc}(0,T;L^{r}(V_{i}))}.\ea\ee
Thus, to prove \eqref{p1},  it suffices to estimate the last term in \eqref{b3}. Since 
\bnn \ba(\ti{\n} -\n)\p_{t}\ti{u}_{i}^{\ep}
&=  (\n(x^\ep,t)-\n(x,t)) \int \int u(y,s)\p_{t}\eta_{\ep}(x^\ep-y,t-s)dyds\\
&\le C\frac{|\n(x^\ep,t)-\n(x,t)|}{\ep}\int_{(t-\ep,t+\ep)} \int_{B(x^\ep,\ep)} \frac{|u(y,s)|}{\ep^{4}} dyds\\
&\le C\left(\int_{0}^{1}|\na \n(x+\tau \ep e_{3},t)| d\tau\right) (|u|*\ti{J}_{\ep})\ea\enn
with $\ti{J}_{\ep}(x,t)=\frac{1}{\ep^{4}} \textbf{1}_{B(0,\ep)}(x^\ep,t),$
then,
  \be\la{p4}\ba \|(\ti{\n} -\n)\p_{t}\ti{u}_{i}^{\ep}\|_{L^{r}_{loc}(0,T;L^{r}(V_{i}))}
  & \le    C\||u|*J_{\ep}\|_{L^{r_{2}}_{loc}(0,T;L^{r_{2}}(V_{i}))}\|\na \n\|_{L^{r_{1}}(V_{i}\times(0,T))} \\
  &\le  C\|u\|_{L^{r_{2}}([0,T]\times \om)} \|\na \n\|_{L^{r_{1}}([0,T]\times \om)}.\ea\ee 
Therefore,  combining  \eqref{b3} with \eqref{p4},  we get   \eqref{p1}. 

The  argument for  \eqref{p2}  goes  similarly, and 
 \eqref{p3}  follows from \eqref{p1},   \eqref{p2} by a density arguments.
 \end{proof}
  
\begin{lemma}\la{lemma2}  Let  $1\le r,\,  r_{1},\,r_{2}<\infty,\,\,  \frac{1}{r_{1}}+\frac{1}{r_{2}}=\frac{1}{r}$,  and   $f\in L^{r_{1}},\,g\in L^{r_{2}}$.  Then 
\be\la{q1} (fg)^{\ep}-fg^{\ep} \rightarrow 0\quad {\rm in}\quad L^{r}_{loc}([0,T]\times \om)\quad{\rm as}\,\,\,\ep\rightarrow0.\ee
\end{lemma}
\begin{proof} The   H\"older inequality gives 
\bnn\ba \  \left| (f g)^{\ep}-fg^{\ep} \right| &= \left| \int\int (f(y,s)-f(x,t))g(y,s)\eta_{\ep}(x-y,t-s)dyds \right| \\
&\le \left(\frac{1}{\ep^{4}}\int_{t-\ep}^{t+\ep}\int_{B(x,\ep)}|f(x,t)-f(y,s)|^{r_{1}}\right)^{\frac{1}{r_{1}}}\left(\frac{1}{\ep^{4}}\int_{t-\ep}^{t+\ep}\int_{B(x,\ep)}|g|^{r_{2}}\right)^{\frac{1}{r_{2}}}\\
&\le  \left(\frac{1}{\ep^{4}}\int_{B(x,\ep)}|f(x,t)-f(y,s)|^{r_{1}} \right)^{\frac{1}{r_{1}}}\left(  |g|^{r_{2}}*J_{\ep}\right)^{\frac{1}{r_{2}}},\ea\enn
  with $ J_{\ep}(x,t)=\frac{1}{\ep^{4}}\textbf{1}_{B(0,\ep)}(x,t).$  
 The Lebesgue  Differentiation  Theorem implies that 
\bnn\frac{1}{\ep^{4}}\int_{t-\ep}^{t+\ep}\int_{B(x,\ep)}|f(x,t)-f(y,s)|^{r_{1}} \rightarrow0,\quad {\rm as}\,\,\ep\rightarrow0.\enn  
Noticing that   \bnn\|(fg)^{\ep}-fg^{\ep}\|_{L^{r}_{loc}(\om\times (0,T))}\le \|f\|_{L^{r_{1}}}   \|g\|_{L^{r_{2}}}\le C, 
\enn
    we  obtain \eqref{q1} by   the Dominated Convergence Theorem. 
\end{proof}
   
The following Hardy-type imbedding will be useful for later use.
\begin{lemma}[\cite{ku}]\la{lemma}  Let $p\in[1,\infty)$ and $f\in W_{0}^{1,p}(\om)$.   There is a  constant C which  depends on $p$ and $\om$, such that  
\bnn \left\| \frac{f(x)}{dist (x,\p\om)} \right\|_{ L^{p}(\om)}\le C\|f\|_{W_{0}^{1,p}(\om)},\enn   \end{lemma}

\bigskip

 \section{Proof of Theorem \ref{t1}}\la{sec_proof}
  
In the following,   we will still use the conventions mentioned in Section \ref{sec_prep}.  

Taking the $j$-th component of equations $\eqref{1.1}_{2}$, testing it against $\eta_{\ep}(x^\ep-y,t-s)$, summing up  the expressions, and using  \eqref{ws3},  we deduce for every $(x,t)\in V_{i}\times (\ep,T-\ep)$ with  $i\in \{1,2,\cdot\cdot\cdot,k\}$, 
 \be\la{u3} \p_{t}(\ti{\n}\ti{u})_{i}^{\ep}+\div (\ti{\n}\ti{u} \otimes \ti{u})_{i}^{\ep}+\na (\ti{\n}^{\ga})_{i}^{\ep}-\mu \lap \ti{u}_{i}^{\ep}-(\mu+\lambda)\na \div \ti{u}_{i}^{\ep}=0.\ee   
To explain how we obtain \eqref{u3}, let us  take the term  $\div (\ti{\n}\ti{u} \otimes \ti{u})_{i}^{\ep}$ for example. In fact,   by \eqref{ws3},
 \bnn\ba \div (\ti{\n}\ti{u} \otimes \ti{u}^{j})_{i}^{\ep}(x,t)
 &=\int_{0}^{T}\int_{\mathcal{V}_{i}}\div_{y}(\ti{\n}\ti{u} \otimes \ti{u}^{j})(y,s) \eta_{\ep}(x-y,t-s)dyds\\
 &=\int_{0}^{T}\int_{\mathcal{V}_{i}} (\n u\otimes u^{j} )(y-\ep \vec{n}(x_1),s)\cdot \na_{x}\eta_{\ep}(x-y,t-s)dyds\\
 &=\int_{0}^{T}\int_{\mathcal{V}_{i}-\ep \vec{n}(x_1)} \n u\otimes u^{j} (y,s)\cdot \na_{x}\eta_{\ep}(x^\ep-y,t-s)dyds.\ea\enn

Similarly,  if we test   $\eqref{1.1}_{2}$ against   $\eta_{\ep}(x-y,t-s)$, we infer that for every $(x,t)\in V_{0}\times (\ep,T-\ep)$,
 \be\la{u4}  \p_{t} (\n u)^{\ep} +\div (\n u \otimes u)^{\ep} +\na ( \n^{\ga})^{\ep} -\mu \lap u^{\ep}-(\mu+\lambda)\na \div u^{\ep}=0.\ee
 
Combining  \eqref{u3} with \eqref{u4} implies that    \be\ba\la{u1}&\p_{t}\left(\xi_{0}(\n u)^{\ep} +\sum_{i=1}^{k}\xi_{i} (\ti{\n}\ti{ u})_{i}^{\ep}\right)+\left(\xi_{0}\div (\n u \otimes u)^{\ep}+\sum_{i=1}^{k}\xi_{i} \div (\ti{\n}\ti {u} \otimes \ti{u})_{i}^{\ep}\right)\\
&+\left(\xi_{0}\na (\n^{\ga})^{\ep} +\sum_{i=1}^{k}\xi_{i}\na  (\ti{\n}^{\ga})_{i}^{\ep} \right)\\
&-\left(\xi_{0}(\mu \lap u^{\ep}+(\mu+\lambda)\na \div u^{\ep})+\sum_{i=1}^{k}\xi_{i}(\mu \lap \ti{u}_{i}^{\ep}+(\mu+\lambda)\na \div \ti{u}_{i}^{\ep})\right)=0,\ea\ee
where $\xi_{0}$ and $\xi_{i}$ are given in  \eqref{u2}. 

Next, we fix   small  constants $\tau>0,\,\de>0,$ and define the  cut-off functions $\psi_{\tau}(t)\in C_{0}^{1}((\tau,T-\tau))$ and $\phi_{\de}(x)\in C_{0}^{1}(\om)$ satisfying
\be\label{phi}\left\{\ba &0\le \phi_{\de}(x)\le 1,\,\,\,\phi_{\de}(x)=1\,\,{\rm if}\,\, x\in \om\,\,\,{\rm and}\,\,\,dist\,(x,\,\p\om)\ge \de,\\
&\phi_{\de}\rightarrow 1\,\,\,{\rm as}\,\,\,\de\rightarrow0,\quad {\rm and}\,\,\,|\na \phi_{\de}|\le \frac{2}{dist(x,\p\om)}.\ea\right.\ee
This way $\psi_{\tau} \phi_{\de} [u]^{\ep}$ is a legitimate test function, where $[u]^{\ep}$ is defined in \eqref{ws15}. Multiplying   \eqref{u1}   by $\psi_{\tau} \phi_{\de} [u]^{\ep}$ and integrating  it over $\om\times (0,T)$ leads to   
\be\la{2.1}  \ba
&\int_{0}^{T} \int_{\om}\psi_{\tau}\phi_{\de}[u]^{\ep}\p_{t}\left(\xi_{0}(\n u)^{\ep}+\sum_{i=1}^{k}\xi_{i}  (\ti{\n}\ti{ u})_{i}^{\ep}\right) \\
&+\int_{0}^{T}\int_{\om} \psi_{\tau} \phi_{\de}[u]^{\ep}\left(\xi_{0}\div (\n u \otimes u)^{\ep}+\sum_{i=1}^{k}\xi_{i} \div (\ti{\n}\ti {u} \otimes \ti{u})_{i}^{\ep}\right)\\
&+\int_{0}^{T}\int_{\om} \psi_{\tau}\phi_{\de}[u]^{\ep}\left(\xi_{0}\na (\n^{\ga})^{\ep} +\sum_{i=1}^{k}\xi_{i}\na  (\ti{\n}^{\ga})_{i}^{\ep} \right)\\
&-\int_{0}^{T}\int_{\om}\psi_{\tau} \phi_{\de}[u]^{\ep}\left(\xi_{0}(\mu \lap u^{\ep}+(\mu+\lambda)\na \div u^{\ep})+\sum_{i=1}^{k}\xi_{i}(\mu \lap \ti{u}_{i}^{\ep}+(\mu+\lambda)\na \div \ti{u}_{i}^{\ep})\right)=0.\ea\ee

In the rest,  we will calculate the terms in \eqref{2.1}  one by one,  and send   $\ep,\,\de,\,\tau$ to zero in the following three steps.

\subsection{Step 1:  $\ep$-limit for \eqref{2.1}}
  
 \begin{lemma}\la{l1} For fixed $\tau$ and $\de$,  the first    two terms in \eqref{2.1} satisfy   \be\ba\la{liang1}& \lim_{\ep\rightarrow0}\int_{0}^{T}\int_{\om} \psi_{\tau} \phi_{\de}[u]^{\ep}\p_{t}\left(\xi_{0}(\n u)^{\ep}+\sum_{i=1}^{k}\xi_{i}  (\ti{\n}\ti{ u})_{i}^{\ep}\right)\\
 &\quad+ \lim_{\ep\rightarrow0} \int_{0}^{T}\int_{\om} \psi_{\tau} \phi_{\de}[u]^{\ep}\left(\xi_{0}\div (\n u \otimes u)^{\ep}+\sum_{i=1}^{k}\xi_{i} \div (\ti{\n}\ti {u} \otimes \ti{u})_{i}^{\ep}\right)\\
&=-\frac{1}{2}\int_{0}^{T}\int_{\om} \psi_{\tau}' \phi_{\de}\n  |u|^{2}-\frac{1}{2}\int_{0}^{T}\int_{\om} \psi_{\tau}  \n u \cdot \na \phi_{\de}   |u|^{2}.\ea\ee
\end{lemma}

\begin{proof}
Firstly, we have
\be\la{m1}\ba &\int_{0}^{T}\int_{\om} \psi_{\tau} \phi_{\de}[u]^{\ep}\p_{t}\left(\xi_{0}(\n u)^{\ep}+\sum_{i=1}^{k}\xi_{i}  (\ti{\n}\ti{ u})_{i}^{\ep}\right)\\
&=\int_{0}^{T}\int_{\om}\psi_{\tau} \phi_{\de}[u]^{\ep}\p_{t}\left(\xi_{0}(\n u)^{\ep}+\sum_{i=1}^{k}\xi_{i} (\ti{\n}\ti{ u})_{i}^{\ep}-\n [u]^{\ep}\right)+\int_{0}^{T}\int_{\om} \psi_{\tau}\phi_{\de}[u]^{\ep} \p_{t}(\n [u]^{\ep}) \\ 
&=: I_{1}+\int_{0}^{T}\int_{\om} \psi_{\tau}\phi_{\de} [u]^{\ep} \p_{t}(\n [u]^{\ep}) .\ea\ee

Let us show 
\be\la{p7}\ba & \lim_{\ep\rightarrow0}I_{1}= 0.\ea\ee
The definition of $[u]^{\ep}$ in   \eqref{ws15} implies    
\bnn \n [u]^{\ep}=\n \left(\xi_{0}u^{\ep}+\sum_{i=1}^{k}\xi_{i}\ti{u}^{\ep}_{i} \right)=\xi_{0}\n u^{\ep}+\sum_{i=1}^{k}\xi_{i}\n \ti{u}^{\ep}_{i}.\enn
This,  along  with  \eqref{u2},  \eqref{p1},   \eqref{p5}, 
implies that 
\be\la{ws1}\ba   |I_{1}| &\le  C\int_{\tau}^{T-\tau}\|u\|_{L^{q}} \left(\|\xi_{0}\p_{t} ( (\n u)^{\ep}-\n u^{\ep})\|_{L^{\frac{q}{q-1}}(\om)}
+\sum_{i=1}^{k}\|\xi_{i}\p_{t}( (\ti{\n}\ti{ u})_{i}^{\ep}-\n  \ti{u}_{i}^{\ep})\|_{L^{\frac{q}{q-1}}(\om)}\right)\\
&\le C\int_{\tau}^{T-\tau}\|u\|_{L^{q}} \left(\|\p_{t} ( (\n u)^{\ep}-\n u^{\ep})\|_{L^{\frac{q}{q-1}}(V_{0})}
+\sum_{i=1}^{k}\|\p_{t}( (\ti{\n}\ti{ u})_{i}^{\ep}-\n  \ti{u}_{i}^{\ep})\|_{L^{\frac{q}{q-1}}(V_{i})}\right)\\
&\le C\int_{0}^{T}\|u\|_{L^{q}}^{2}\left(\||\p_{t}\n \|_{L^{\frac{q}{q-2}}}+\|\na \n\|_{L^{\frac{q}{q-2}}}\right).\ea\ee

On the other hand,  
it follows from     \eqref{1.12},  \eqref{1.5a},  \eqref{1.5} that 
\be\la{o}  \n_{t}= -\left(\n \div u+2\sqrt{\n} u \cdot \na \sqrt{\n}\right)\in L^{2}(0,T;L^{2})+L^{p}(0,T;L^{\frac{2q}{q+2}}).\ee
Therefore,  from \eqref{o}, \eqref{1.5a},  \eqref{1.5}   we deduce
\bnn\ba
|I_{1}|\le C\int_{0}^{T}\left( \|u\|_{L^{q}}^{p}+ \||\p_{t}\n \|_{L^{\frac{q}{q-2}}}^{\frac{p}{p-2}}+\|\na \n\|_{L^{\frac{q}{q-2}}}^{\frac{p}{p-2}} \right) \le C,
\ea\enn provided that $p\ge 4,\,\,q\ge6.$

Furthermore,   with    \eqref{o} and  \eqref{1.5a},  from  Lemma \ref{lem} and Corollary \ref{co}  we obtain 
$$\p_{t}\left((\n u)^{\ep} -\n u^{\ep}\right) {\rm in }\ L^{\frac{2p}{p+2}}_{loc}\left(0,T; L^{\frac{2q}{q+4}}(V_{0})\right), $$
$$\p_{t}\left((\ti{\n} \ti{u})_{i}^{\ep} -\n \ti{u}_{i}^{\ep}\right)\rightarrow 0 \ {\rm in }\ L^{\frac{2p}{p+2}}_{loc}\left(0,T; L^{\frac{2q}{q+4}}(V_{i})\right).$$
This concludes   \eqref{p7}.
 
Secondly,     the  convection term can be treated as
  \be\la{m2}\ba
&\int_{0}^{T}\int_{\om} \psi_{\tau} \phi_{\de}[u]^{\ep}\left(\xi_{0}\div (\n u \otimes u)^{\ep}+\sum_{i=1}^{k}\xi_{i} \div (\ti{\n}\ti {u} \otimes \ti{u})_{i}^{\ep}\right)\\
&=\int_{0}^{T}\int_{\om}\psi_{\tau}\phi_{\de}[u]^{\ep}\left(\xi_{0}\div (\n u \otimes u)^{\ep}+\sum_{i=1}^{k}\xi_{i} \div (\ti{\n}\ti{ u} \otimes \ti{u})_{i}^{\ep}-{\rm div}(\n u\otimes [u]^{\ep})\right)\\
&\quad+\int_{0}^{T}\int_{\om}\psi_{\tau} \phi_{\de}[u]^{\ep} {\rm div}(\n u\otimes [u]^{\ep})\\
&=: I_{2}+\int_{0}^{T}\int_{\om}\psi_{\tau} \phi_{\de}[u]^{\ep} {\rm div}(\n u\otimes [u]^{\ep}).\ea\ee
We claim that
\be\la{2.11}\lim_{\ep\rightarrow0}  I_{2}= 0.\ee
In fact,  by \eqref{u2},    
\be\la{ws2}\ba   |I_{2}| 
&\le C\int_{\tau}^{T-\tau}\|u\|_{L^{q}}  \|\div\left( (\n u \otimes u)^{\ep} -\n u \otimes [u]^{\ep}\right)\|_{L^{\frac{q}{q-1}}(V_{0})}\\ 
&\quad+ C\sum_{i=1}^{k}\int_{\tau}^{T-\tau}\|u\|_{L^{q}}   \|\div \left(\left(\ti{\n}\ti{u}\otimes\ti{u}\right)_{i}^{\ep}-\n u\otimes [u]^{\ep}\right)\|_{L^{\frac{q}{q-1}}(V_{i})}\\
&=: I_{21}+\sum_{i=1}^{k}I_{22}^{i}.
\ea\ee
Making use of \eqref{p5} and  \eqref{1.5a},  one has 
 \bnn \|\na(\n u)\|_{L^{\frac{q}{q-2}}}\le C\|u\|_{H^{1}}(1+\|\na\sqrt{\n}\|_{L^{2}})\le C \| \na u\|_{L^{2}}. \enn 
Thus, 
\be\la{p11}\ba & |I_{21}|\\
&\le \int_{\tau}^{T-\tau}\|u\|_{L^{q}} \left(\|\div\left( (\n u \otimes u)^{\ep} -\n u\otimes u^{\ep}\right)\|_{L^{\frac{q}{q-1}}(V_{0})}+\|\div\left( \n u \otimes (u^{\ep} - [u]^{\ep})\right)\|_{L^{\frac{q}{q-1}}(V_{0})}\right)\\
&\le C \int_{0}^{T}\|u\|_{L^{q}}^{2}\|\na (\n u)\|_{L^{\frac{q}{q-2}}}+\int_{\tau}^{T-\tau}\|u\|_{L^{q}}  \|\div\left( \n u \otimes (u^{\ep} - [u]^{\ep})\right)\|_{L^{\frac{q}{q-1}}(V_{0})}\\
&\le C \int_{0}^{T}\left(\|u\|_{L^{q}}^{4}+\|\na u\|_{L^{2}}^{2}\right)+\int_{\tau}^{T-\tau}\|u\|_{L^{q}}  \|\div\left( \n u \otimes (u^{\ep} - [u]^{\ep})\right)\|_{L^{\frac{q}{q-1}}(V_{0})},\ea\ee
Notice that 
\be\la{p11a}\ba 
& \int_{\tau}^{T-\tau}\|u\|_{L^{q}} \|\div\left( \n u \otimes (u^{\ep} - [u]^{\ep})\right)\|_{L^{\frac{q}{q-1}}(V_{0})}\\
&\ \le C\int_{\tau}^{T-\tau}\|u\|_{L^{q}}\left(\|\na(\n u)\|_{L^{\frac{q}{q-2}}} \|u^{\ep}-[u]^{\ep}\|_{L^{q}(V_{0})} +\|\n u\|_{L^{q}}\|\na (u^{\ep}-[u]^{\ep})\|_{L^{\frac{q}{q-2}}(V_{0})}\right)\\
&\ \le  C\int_{\tau}^{T-\tau}\|u\|_{L^{q}} \left( \|u\|_{H^{1}}\|u^{\ep}-[u]^{\ep}\|_{L^{q}(V_{0})} +\|\na (u^{\ep}-[u]^{\ep})\|_{L^{2}(V_{0})}\|u\|_{L^{q}}\right)\\
&\ \le  C\left(\int_{\tau}^{T-\tau} \|u^{\ep}-[u]^{\ep}\|_{L^{q}(V_{0})}^{p}+\|\na (u^{\ep}-[u]^{\ep})\|_{L^{2}(V_{0})}^{2}\right)^{\frac{1}{2}},\ea\ee
and, owing to  \eqref{1.12}, \eqref{1.5},     \eqref{p9},
\be\la{6y}\lim_{\ep\rightarrow0} \left(\|u^{\ep}-[u]^{\ep}\|_{L^{p}_{loc}(0,T;L^{q}(V_{0}))}+\|\na (u^{\ep}-[u]^{\ep})\|_{L^{2}_{loc}(0,T;L^{2}(V_{0}))}\right)=0.\ee
We conclude    from   \eqref{6y} and   \eqref{6} that 
\be\la{p10}\lim_{\ep\rightarrow0} I_{21}=0.\ee
 
By  \eqref{p2},  a similar argument to \eqref{p11}  and  \eqref{p11a} infers that 
\be\la{ws2b}
\ba  & |I_{22}^{i}| \\
&\le  \int_{\tau}^{T-\tau}\|u\|_{L^{q}} \left( \|\div\left( (\ti{\n}\ti{u}\otimes \ti{u})_{i}^{\ep}  -\n u \otimes \ti{u}_{i}^{\ep}\right)\|_{L^{\frac{q}{q-1}}(V_{i})}+\|\div\left( \n u\otimes(\ti{u}_{i}^{\ep}-[u]^{\ep})\right)\|_{L^{\frac{q}{q-1}}(V_{i})}\right)\\
&\le C \int_{0}^{T}\left(\|u\|_{L^{q}}^{4}+\|\na u\|_{L^{2}}^{2}\right)+ C \int_{\tau}^{T-\tau} \left(\|u^{\ep}-[u]^{\ep}\|_{L^{q}(V_{i})}^{p}+\|\na (u^{\ep}-[u]^{\ep})\|_{L^{2}(V_{i})}^{2}\right).\ea\ee
and consequently,  from   \eqref{p3} and   \eqref{6y}   we get  
\bnn\lim_{\ep\rightarrow0} I_{22}^{i}=0.\enn
This together with  \eqref{p10}  implies  \eqref{2.11}.

Next,   by the continuity  equation $\eqref{1.1}_{1}$, a simple computation shows that
 \be\la{y1}\ba &\int_{0}^{T}\int_{\om} \psi_{\tau}\phi_{\de} [u]^{\ep} \p_{t}(\n [u]^{\ep})  +\int_{0}^{T}\int_{\om}\psi_{\tau} \phi_{\de}[u]^{\ep} {\rm div}(\n u\otimes [u]^{\ep})\\
&=-\frac{1}{2}\int_{0}^{T}\int_{\om}\psi_{\tau}' \phi_{\de}\n  |[u]^{\ep}|^{2}-\frac{1}{2}\int_{0}^{T}\int_{\om} \psi_{\tau}\n u \cdot \na \phi_{\de}   |[u]^{\ep}|^{2}\\
&\quad -\int_{0}^{T}\int_{\om} \n \p_{t}\left(\phi_{\de}\psi_{\tau}\frac{|[u]^{\ep}|^{2}}{2} \right) - \n u\cdot \na \left(\phi_{\de}\psi_{\tau}\frac{|[u]^{\ep}|^{2}}{2} \right)\\
&=-\frac{1}{2}\int_{0}^{T}\int_{\om}\psi_{\tau}' \phi_{\de}\n  |[u]^{\ep}|^{2}-\frac{1}{2}\int_{0}^{T}\int_{\om} \psi_{\tau}\n u \cdot \na \phi_{\de}   |[u]^{\ep}|^{2}\\
&\rightarrow -\frac{1}{2}\int_{0}^{T}\int_{\om}\psi_{\tau}' \phi_{\de}\n  |u|^{2}-\frac{1}{2}\int_{0}^{T}\int_{\om} \psi_{\tau} \n u \cdot \na \phi_{\de}   |u|^{2}, \quad {\rm as}\,\,\ep\rightarrow0.\ea\ee
In conclusion,  owing to  \eqref{p7}, \eqref{2.11}, \eqref{y1},   we  get \eqref{liang1} from \eqref{m1} and \eqref{m2}.
\end{proof}

\begin{lemma} For fixed $\tau$ and $\de$,  the pressure  term  in \eqref{2.1} satisfies 
 \la{l2}\be\ba \la{liang2}&\lim_{\ep\rightarrow0}\int_{0}^{T}\int_{\om} \psi_{\tau}\phi_{\de} [u]^{\ep}\left(\xi_{0}\na (\n^{\ga})^{\ep} +\sum_{i=1}^{k}\xi_{i}\na  (\ti{\n}^{\ga})_{i}^{\ep} \right)= \int_{0}^{T}\int_{\om} \psi_{\tau} \phi_{\de}u\cdot \na \n^{\ga}.\ea\ee
\end{lemma}
\begin{proof} 
 Owing to   \eqref{1.5a},  we have  \be\la{k1}\na \n^{\ga}\in  L^{2}(0,T;L^{2}).\ee
We write  \be\ba\la{2.8}&\int_{0}^{T}\int_{\om} \psi_{\tau} \phi_{\de} [u]^{\ep}\left(\xi_{0}\na (\n^{\ga})^{\ep} +\sum_{i=1}^{k}\xi_{i}\na  (\ti{\n}^{\ga})_{i}^{\ep} \right)\\
&=\int_{0}^{T}\int_{\om} \psi_{\tau} \phi_{\de} [u]^{\ep}\left(\xi_{0}\na (\n^{\ga})^{\ep} +\sum_{i=1}^{k}\xi_{i}\na (\ti{\n}^{\ga})_{i}^{\ep} -\na  \n^{\ga} \right)+\int_{0}^{T}\int_{\om} \psi_{\tau} \phi_{\de}   [u]^{\ep}\cdot \na\n^{\ga}\\
&=: I_{3}+ \int_{0}^{T}\int_{\om} \psi_{\tau} \phi_{\de} [u]^{\ep}\cdot \na \n^{\ga},\ea\ee
where the last integral makes sense due to \eqref{1.12} and \eqref{k1}.

Observe from \eqref{ws16} and \eqref{ws10} that  \be\la{p13} \lim_{\ep\rightarrow0}\left(\|\na ((\n^{\ga})^{\ep} - \n^{\ga})\|_{L^{2}_{loc}(0,T;L^{2}(V_{0}))}+\|\na ((\ti{\n}^{\ga})_{i}^{\ep} - \n^{\ga})\|_{L^{2}_{loc}(0,T;L^{2}(V_{i}))}\right)=0,\ee   and hence,  \be\ba\la{2.8q} &\lim_{\ep\rightarrow0}|I_{3}|\\
&\le C\lim_{\ep\rightarrow0} \|u\|_{L^{2} (0,T;L^{2}(\om))}\left(\|\xi_{0}\na ((\n^{\ga})^{\ep} - \n^{\ga} ) +\sum_{i=1}^{k}\xi_{i}\na ((\ti{\n}^{\ga})_{i}^{\ep} -  \n^{\ga} )\|_{L^{2}_{loc}(0,T;L^{2}(\om))}\right)\\
&\le  C\lim_{\ep\rightarrow0}\left(\|\na ((\n^{\ga})^{\ep} - \n^{\ga})\|_{L^{2}_{loc}(0,T;L^{2}(V_{0}))}+\|\na ((\ti{\n}^{\ga})_{i}^{\ep} - \n^{\ga})\|_{L^{2}_{loc}(0,T;L^{2}(V_{i}))}\right)\\
&=0.\ea\ee

Taking \eqref{2.8q}, \eqref{p9a} into accounts, we  take $\ep\rightarrow 0$ in \eqref{2.8}  and complete the proof for Lemma \ref{l2}. 
\end{proof}

 \begin{lemma}\la{l3}For fixed $\tau$ and $\de$,  the diffusion terms in \eqref{2.1} satisfy \be\ba\la{liang3}&\lim_{\ep\rightarrow0}\int_{0}^{T}\int_{\om}\psi_{\tau} \phi_{\de}[u]^{\ep}\left(\xi_{0}(\mu \lap u^{\ep}+(\mu+\lambda)\na \div u^{\ep})+\sum_{i=1}^{k}\xi_{i}(\mu \lap \ti{u}_{i}^{\ep}+(\mu+\lambda)\na \div \ti{u}_{i}^{\ep})\right)\\
 &=-\int_{0}^{T}\int_{\om}\psi_{\tau}\phi_{\de} \left (\mu|\na u|^{2}+(\mu+\lambda) (\div u)^{2}\right) -\int_{0}^{T}\int_{\om}\psi_{\tau} \na \phi_{\de} \left(  \mu  u\na u+(\mu+\lambda)u \div u\right).\ea\ee
\end{lemma}
\begin{proof} 
We see that
 \be\la{q14} \ba & \mu\int_{0}^{T}\int_{\om}\psi_{\tau}\phi_{\de}  [u]^{\ep}\left(\xi_{0}  \lap u^{\ep}+\sum_{i=1}^{k}\xi_{i}  \lap \ti u^{\ep}\right)\\
 &= \mu\int_{0}^{T}\int_{\om}\psi_{\tau}\phi_{\de}  [u]^{\ep}\left(\xi_{0}  \lap u^{\ep}+\sum_{i=1}^{k}\xi_{i}  \lap \ti u^{\ep}- \lap[u]^{\ep}\right) +\mu\int_{0}^{T}\int_{\om}\psi_{\tau} \phi_{\de} [u]^{\ep} \lap [u]^{\ep}  \\
 &=: I_{4}-\mu\int_{0}^{T}\int_{\om}\psi_{\tau}\na \phi_{\de} [u]^{\ep} \na [u]^{\ep} - \mu\int_{0}^{T}\int_{\om}\phi_{\de}  \psi_{\tau} \na  [u]^{\ep}: \na [u]^{\ep}.\ea\ee

We compute
\bnn\ba I_{4} &=\int_{0}^{T}\int_{\om}  \psi_{\tau} \phi_{\de} [u]^{\ep}\left(\xi_{0}\lap (u^{\ep}-[u]^{\ep})+\sum_{i=1}^{k}\xi_{i}\lap (\ti{u}_{i}^{\ep}-[u]^{\ep})\right)\\
 &= - \int_{0}^{T}\int_{\om} \psi_{\tau}\left(\na\phi_{\de}[u]^{\ep}\xi_{0}+\phi_{\de}\div [u]^{\ep}\xi_{0}+\phi_{\de}[u]^{\ep}\na \xi_{0}\right) \na (u^{\ep}-[u]^{\ep})\\
 &\quad -\sum_{i=1}^{k}\int_{0}^{T}\int_{\om} \psi_{\tau}\left(\na\phi_{\de}[u]^{\ep}\xi_{i}+\phi_{\de}\div [u]^{\ep}\xi_{i}+\phi_{\de} [u]^{\ep}\na \xi_{i}\right) \na (\ti{u}_{i}^{\ep}-[u]^{\ep}). \ea\enn  
Thus
\begin{equation*}
|I_4| \le C(\de) \left(\|\na u^{\ep}-\na [u]^{\ep}\|_{L^{2}_{loc}(0,T;L^{2}(V_{0}))}+\sum_{i=1}^{k}\|\na \ti{u}_{i}^{\ep}-\na [u]^{\ep}\|_{L^{2}_{loc}(0,T;L^{2}(V_{i}))}\right).
\end{equation*}
In view of   \eqref{1.12},  \eqref{p9a}, \eqref{p9},  it yields from \eqref{q14} 
that 
\bnn \ba & \mu\int_{0}^{T}\int_{\om}\psi_{\tau}\phi_{\de}  [u]^{\ep}\left(\xi_{0}  \lap u^{\ep}+\sum_{i=1}^{k}\xi_{i}  \lap u^{\ep}\right)\\
&\rightarrow  -\mu\int_{0}^{T}\int_{\om}\psi_{\tau}\na \phi_{\de} u \na u- \mu\int_{0}^{T}\int_{\om}\phi_{\de}  \psi_{\tau} \na u: \na u,\ea\enn
as  $\ep\rightarrow0$.  
Applying the similar arguments  for  other diffusion terms,  we get \eqref{liang3}, and hence the lemma is proved.
\end{proof}

In summary,     Lemma \ref{l1}-\ref{l3}  and equality \eqref{2.1}  imply that  
 \be\la{m3}\ba &-\frac{1}{2}\int_{0}^{T}\int_{\om} \psi_{\tau}' \phi_{\de}\n  |u|^{2}-\frac{1}{2}\int_{0}^{T}\int_{\om} \psi_{\tau}  \n u \cdot \na \phi_{\de}   |u|^{2}\\
 &-\int_{0}^{T}\int_{\om} \psi_{\tau} \phi_{\de}\n^{\ga}\div u -\int_{0}^{T}\int_{\om} \psi_{\tau} u\cdot \na \phi_{\de} \n^{\ga}\\
 &-\int_{0}^{T}\int_{\om}\psi_{\tau}\phi_{\de} \left (\mu|\na u|^{2}+(\mu+\lambda) (\div u)^{2}\right) -\int_{0}^{T}\int_{\om}\psi_{\tau} \na \phi_{\de} \left(  \mu  u\na u+(\mu+\lambda)u \div u\right)=0.\ea\ee

Next we consider taking the $\delta$-limit. 
\subsection{Step 2:  $\de$-limit for \eqref{m3}}
 
Thanks to  \eqref{1.12},    \eqref{1a},    \eqref{1.5a}, \eqref{1.5}, and Lemma \ref{lemma}, it follows that 
 \be\la{m3q}\ba & -\frac{1}{2}\int_{0}^{T}\int_{\om} \psi_{\tau}  \n u \cdot \na \phi_{\de}   |u|^{2}   -\int_{0}^{T}\int_{\om} \psi_{\tau} u\cdot \na \phi_{\de} \n^{\ga} -\int_{0}^{T}\int_{\om}\psi_{\tau} \na \phi_{\de} \left(  \mu  u\na u+(\mu+\lambda)u \div u\right)\\
 &\le C\big\||u||\na \phi_{\de}|\big\|_{L^{2}(0,T;L^{2})} \left(\int_{0}^{T}\int_{\{x:\,dist(x,\,\p\om)<\de\}} |u|^{4}+\n^{2}+|\na u|^{2} \right)^{\frac{1}{2}}\\
  &\le C\left\|\frac{u}{dist(x,\,\p\om)}\right\|_{L^{2}(0,T;L^{2})} \left(\int_{0}^{T}\int_{\{x:\,dist(x,\,\p\om)<\de\}} |u|^{4}+\n^{2}+|\na u|^{2} \right)^{\frac{1}{2}}\\
 &\rightarrow 0,\quad {\rm as}\,\,\,\de\rightarrow0.\ea\ee 
By  \eqref{m3q}, taking  $\de\rightarrow0$ in \eqref{m3}, 
   \be\la{2.3g}  \ba  
 \frac{1}{2}\int_{0}^{T}\int_{\om} \psi_{\tau} '\n |u|^{2}+\int_{0}^{T}\int_{\om} \psi_{\tau}   \n^{\ga} \div u
+\int_{0}^{T}\int_{\om}\psi  \na u: \mathbb{S}=0.\ea\ee 

On the other hand,   it follows from $\eqref{1.1}_{1}$, \eqref{1.12}, \eqref{1.5a} that  \bnn\ba &\int_{0}^{T}\int_{\om} \psi_{\tau}   \n^{\ga} \div u=\int_{0}^{T}\int_{\om}\psi_{\tau} \n^{\ga-1}(\n_{t}+u\cdot\na \n )=\frac{1}{\ga-1}\int_{0}^{T}\int_{\om} \psi_{\tau} '  \n^{\ga}.  \ea\enn  Thus, \eqref{2.3g} becomes 
 \be\la{2.3f}  \ba  
 \frac{1}{2}\int_{0}^{T}\int_{\om} \psi_{\tau}'\n |u|^{2}+\frac{1}{\ga-1}\int_{0}^{T}\int_{\om} \psi_{\tau}'  \n^{\ga}
+\int_{0}^{T}\int_{\om}  \psi_{\tau}  \left(\mu|\na u|^{2} +(\mu+\lambda)|\div u|^{2}\right)=0.\ea\ee 

Denote
\begin{equation}\label{defnED}
E(t) := \int_{\om} \left( {1\over2} \n |u|^2 + {\n^\ga \over \ga - 1} \right) \,dx, \ D(t) := \int_{0}^{t}\int_{\om} \left(\mu|\na u|^{2} +(\mu+\lambda)|\div u|^{2}\right)\,dxds.
\end{equation}
Then \eqref{2.3f} implies that
\begin{equation*}
(E - D)' = 0, \quad \rm{ in } \quad \mathcal{D}'((0,T)). 
\end{equation*}

\subsection{Step 3: Global energy balance}

Finally we will obtain the exact energy equality on the whole time interval $[0,T]$. First we note that $D\in C([0,T])$. Second, we see that an approximation argument shows that \eqref{2.3f} remains valid for functions $\psi_\tau$ belonging only to $W^{1,\infty}$ rather than $C^1$.


It follows from $\eqref{1.1}_{1}$ that for any $\a\ge \frac{1}{2}$,
\bnn \p_{t}(\n^{\a})=-\a \n^{\a}\div u-2\a \n^{\a-\frac{1}{2}} u\cdot\na \sqrt{\n},\enn
which, together with \eqref{1.12} and \eqref{1.5a}, implies 
\bnn \n^{\a}\in  L^{\infty}(0,T; H^{1}(\om)),\quad \p_{t}(\n^{\a})\in L^{2}(0,T; L^{\frac{3}{2}}(\om)).\enn   
   Hence, by the Aubin--Lions  Lemma (cf.\cite{sim}), 
\begin{equation*} \n^{\a}\in C([0,T]; L^{r}(\om)),\quad (r<6).\end{equation*}
In particular we know that 
\begin{equation}\label{rhogamma}
\rho^\gamma \in C([0,T]; L^{1}(\om)).
\end{equation}
 
In a similar way, 
\be\la{m13} \n u\in L^{\infty}(0,T; L^{2}(\om)) \cap  H^{1}(0,T; W^{-1,1}(\om))\hookrightarrow  C([0,T];\,\, L^{2}_{weak}(\om)).\ee  
 
Recalling   \eqref{1.8} and the \eqref{rhogamma},   we have 
\be\la{m14}\ba 0&\le \overline{\lim_{t\rightarrow 0}}\int_{\om}|\sqrt{\n} u-\sqrt{\n_{0}} u_{0}|^{2}\,dx\\
&=2\overline{\lim_{t\rightarrow 0}}\left(\int_{\om}\left(\frac{1}{2}\n |u|^{2}+\frac{1}{\ga-1}\n^{\ga}\right)\,dx-\int_{\om}\left(\frac{1}{2}\n_{0}|u_{0}|^{2}+\frac{1}{\ga-1}\n_{0}^{\ga}\right)\,dx\right)\\
&\quad +2\overline{\lim_{t\rightarrow 0}}\left(\int_{\om} \sqrt{\n_{0}} u_{0}(\sqrt{\n_{0}} u_{0}-\sqrt{\n} u)\,dx+\frac{1}{\ga-1}\int_{\om}( \n_{0}^{\ga}-\n^{\ga})\,dx\right)\\
& \le 2\overline{\lim_{t\rightarrow 0}} \int_{\om} \sqrt{\n_{0}} u_{0}(\sqrt{\n_{0}} u_{0}-\sqrt{\n} u)\,dx,\ea\ee
and furthermore,  
\be\la{m15}\begin{split} \overline{\lim_{t\rightarrow 0}} \int_{\om} \sqrt{\n_{0}} u_{0}(\sqrt{\n_{0}} u_{0}-\sqrt{\n} u)\,dx & =\overline{\lim_{t\rightarrow 0}} \int_{\om} u_{0} ( \n_{0} u_{0}-\n u)\,dx +\overline{\lim_{t\rightarrow 0}} \int_{\om} u_{0} \sqrt{\n} u(\sqrt{\n}-\sqrt{\n_{0}})\,dx \\
& =0,\end{split}\ee
owing to    
\eqref{m13},   \eqref{rhogamma} and  \eqref{1.5s}.  Therefore,  from 
 \eqref{m14},  \eqref{m15}, and \eqref{m13}  we deduce that
 \be\label{contat0}
 (\sqrt{\rho}u)(t) \to (\sqrt{\rho}u)(0) \quad \text{ strongly in }\ L^2(\om) \ \ \text{as } t \to 0^+.
 \ee

Similarly, one has the right temporal continuity of $\sqrt{\rho} u$ in $L^2$, that is, for any $t_0\ge 0$,
 \be\label{rightcont}
 (\sqrt{\rho}u)(t) \to (\sqrt{\rho}u)(t_0) \quad \text{ strongly in }\ L^2(\om) \ \ \text{as } t \to t_0^+.
 \ee
 
Now for $t_0>0$, we choose some positive $\tau$ and $\alpha$ such that $\tau + \alpha < t_0$ and define the following test function 

\noindent\begin{minipage}{.55\textwidth}
\begin{equation*}
\psi_\tau(t) = \left\{\begin{array}{cl}
0, & 0 \le t \le \tau,\\
\displaystyle{t-\tau \over \alpha}, & \tau \le t \le \tau+\alpha,\\
1, & \tau+\alpha \le t \le t_0, \\
\displaystyle{t_0 - t \over \alpha}, & t_0 \le t \le t_0+ \alpha,\\
0, & t_0 + \alpha \le t.
\end{array}\right.
\end{equation*}
\end{minipage}
\noindent\begin{minipage}{.4\textwidth}
\vspace{.25cm}
\includegraphics[scale=0.65]{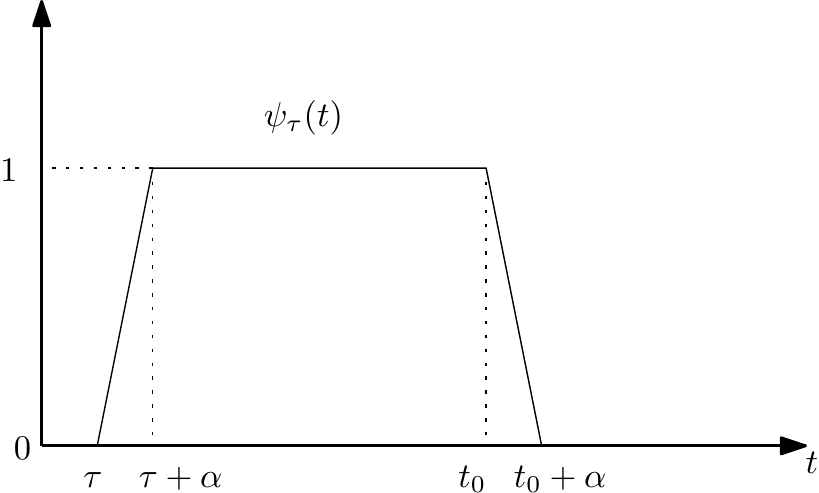}
\vspace{.15cm}
\end{minipage}

Substituting  the above test function into \eqref{2.3f} we obtain that
\begin{align*}
& {1\over \alpha}\int^{\tau+\alpha}_{\tau}\int_{\om} \left( {1\over2}\rho|u|^2 + {\rho^\gamma \over \gamma - 1} \right)\,dxds - {1\over\alpha} \int^{t_0+\alpha}_{t_0} \int_{\om} \left( {1\over2}\rho|u|^2 + {\rho^\gamma \over \gamma - 1} \right)\,dxds \\
= & -\int^{t_0+\alpha}_{\tau} \int_\om \psi_\tau \left( \mu|\na u|^{2} +(\mu+\lambda)|\div u|^{2} \right)\,dxds = D(\tau) - D(t_0 + \alpha),
\end{align*}
where $D$ is defined in \eqref{defnED}.

Sending $\alpha \to 0$, using the right continuity of $\sqrt{\rho}u$ in $L^2$ and the continuity of $\rho^\gamma$ in $L^1$ (cf. \eqref{rhogamma} and \eqref{rightcont}), and the continuity of $D(t)$, one has
\begin{equation*}
E(\tau) - E(t_0) = D(\tau) - D(t_0).
\end{equation*}
Finally sending $\tau \to 0$, from \eqref{contat0} we have 
$$(E - D)(t_0) = (E-D)(0),$$
which is exactly \eqref{energy}, and hence we complete the proof of Theorem \ref{t1}. 
 
 \bigskip
 
\appendix

\section{Application to the incompressible Navier--Stokes equations}

In this appendix, we apply our global approximation method to the incompressible Navier--Stokes equations posed on a bounded domain:
\begin{equation}\label{1.1fu}
\left\{\ba
& u_t +{\rm div}(u\otimes u )+\na P-\nu\lap u=0, \\
& {\rm div} u=0,\\
& u|_{\partial \om} = 0, \\
& u(x,0) = u_0(x), \quad \div u_0 = 0. 
\ea \right. 
\end{equation}

It is well-known that the Leray-Hopf weak solution $u$ to \eqref{1.1fu} satisfies 
\begin{equation}\label{1.12fu}
u \in L^{\infty}\left(0,T;  L^{2}(\om)\right) \cap L^{2}\left(0,T;H_{0}^{1}(\om) \right)
\end{equation}  
and the following energy inequality
\be\la{1.8fu} \int_{\om} \frac{1}{2} |u|^{2} dx+\nu \int_{0}^{t}\int_{\om}|\na u|^{2} dxds\le  \int_{\om} \frac{1}{2} |u_{0}|^{2} dx.\ee

We first recall a result of \cite[Theorem 1]{mauro} regarding the pressure field associated to the Leray--Hopf solution of \eqref{1.1fu}.
\begin{theorem}(\cite[Theorem 1]{mauro})\label{thm_mauro}
Assume that $\om$ is  an  open, bounded domain with   $C^{2}$ boundary $\p \om$, and $u$ is a Leray--Hopf solution of \eqref{1.1fu}. Then there exists a pressure field $P \in L^r(0,T; W^{1,s}(\om))$ with
\be\la{range}
{3\over s} + {2\over r} = 4, \qquad {4\over 3} < s < {3\over2},
\ee
such that for all $\varphi \in C^\infty_0((0, T)\times \om)$,
\begin{equation*}
\int^T_0\int_\om \left( u\cdot \partial_t\varphi + u\otimes u : \nabla \varphi + P \div\varphi + \nu u \cdot \Delta \varphi \right)\,dxdt = 0.
\end{equation*}
\end{theorem}
\begin{remark}
The additional smoothness of the boundary is assumed to ensure the existence of the Leray--Hopf weak solution of \eqref{1.1fu}. 
\end{remark}
\begin{remark}
Another important implication of Theorem \ref{thm_mauro} is that it allows us to use the test function $\psi_\tau\phi_\delta [u]^\varepsilon$ we introduced in the previous sections which is not solenoidal to test against the incompressible Navier-Stokes equations. 
\end{remark}
 
The main result concerning energy equality of \eqref{1.1fu} is 
\begin{theorem}\la{t2} Assume that $\om$ is  an  open, bounded domain with   $C^{2}$ boundary $\p \om$,  and   $u$ is a Leray-Hopf weak solution of \eqref{1.1fu}. 
Then the equality in \eqref{1.8fu} is achieved, provided that
\be\la{fu2} u\in L^{p}(0,T;L^{q}(\om)),\quad\frac{1}{p}+\frac{1}{q}\le \frac{1}{2},\quad  4\le q,\ee 
and the associated pressure given in Theorem \ref{thm_mauro} satisfies
\be\la{pressure}
P \in L^2(0,T; L^2(\partial \om)).
\ee
\end{theorem}

\begin{remark}
A notable difference between our Theorem \ref{thm_mauro} and the result in \cite{yu2}, as indicated in the Introduction, is that we do not need to assume any additional Besov regularity on the velocity to handle the boundary effects coming from the diffusion term. 
\end{remark}

\begin{remark}
From Theorem \ref{thm_mauro} we know that the pressure field enjoys sufficient regularity to define its trace on the boundary $P \in W^{1-{1\over s},s}(\partial \om) \subset L^s(\partial\om)$ for a.e. $t\in [0, T]$. In fact the fractional Sobolev embedding (see, e.g. \cite{nezza}) further implies that 
\[
P \in L^{\tilde s}(\partial\om) \quad \text{ where } \quad s \le \tilde s \le {2s \over 3 - s}.
\]
From \eqref{range} we see that $\tilde s < 2$. Here we need to assume a bit more integrability of the pressure trace (cf. \eqref{pressure}).
\end{remark}
\begin{remark}
Note that by interpolation we see that $L^2H^1 \cap L^4L^4$ lands in the Onsager-critical Besov spaces $L^3B^{1/3}_{3, r}$ for $1\le r < \infty$. It would be interesting to obtain the energy equality for velocities in the Onsager-critical Besov spaces $L^3B^{1/3}_{3, \infty}$ in the interior, a la Constantin et al. \cite{che}.
\end{remark}

The proof of Theorem \ref{t2} is  a slight modification of  that in Theorem \ref{t1}.  
We only  prove the Lemmas \ref{l4.1}--\ref{l4.2} below to address conditions \eqref{fu2}--\eqref{pressure} and the main  differences. An important ingredient in the argument is the global $L^p$ estimate of the pressure, which is given in the following proposition.
\begin{proposition}\label{prop_pressure}
Let the assumptions of Theorem \ref{t2} hold, then the pressure field satisfies
\be\la{fu6} \|P\|_{L^{2}(\om)}\le C \left(\| u\|^2_{L^{4}(\om)} + \|P|_{\partial\om}\|_{L^2} \right).\ee
\end{proposition}
\begin{proof}
The pressure satisfies a Poisson problem together with certain boundary regularity.
 \[\ba
 & -\Delta P = \div\div(u\otimes u) \quad \text{ in } \ \om, \\
 & P|_{\partial \om} \in L^2.
 \ea
\] 
Using duality and the method of transposition (e.g. \cite[Lemma 2]{esc}) we see that
\[
\|P\|_{L^{2}(\om)}\le C \left( \|(u\otimes u)\|_{L^{2}(\om)} + \|P|_{\partial\om}\|_{L^2} \right)\le C \left(\| u\|^2_{L^{4}(\om)} + \|P|_{\partial\om}\|_{L^2} \right),
\]
completing the proof of the proposition.
\end{proof}

Let $u$ be a Leray--Hopf weak solution to \eqref{1.1fu}.  
The same deduction as \eqref{2.1}  yields    \be\la{2.1fu}  \ba
&\int_{0}^{T} \int_{\om} \psi_{\tau} \phi_{\de}[u]^{\ep}\p_{t}\left(\xi_{0}u^{\ep} +\sum_{i=1}^{k}\xi_{i} \ti{ u}_{i}^{\ep}\right)\\
&+\int_{0}^{T}\int_{\om} \psi_{\tau} \phi_{\de}[u]^{\ep}\left(\xi_{0}\div (u \otimes u)^{\ep}+\sum_{i=1}^{k}\xi_{i} \div (\ti {u} \otimes \ti{u})_{i}^{\ep}\right)\\
&+\int_{0}^{T}\int_{\om} \psi_{\tau}\phi_{\de}[u]^{\ep}\left(\xi_{0}\na P^{\ep} +\sum_{i=1}^{k}\xi_{i}\na  \ti{P}_{i}^{\ep} \right) \\
&-\int_{0}^{T}\int_{\om} \psi_{\tau} \phi_{\de}[u]^{\ep}\left(\xi_{0}\lap u^{\ep}+\sum_{i=1}^{k}\xi_{i}\lap \ti{u}_{i}^{\ep}\right)=0.\ea\ee

\begin{lemma} \la{l4.1} The convection term in \eqref{2.1fu} satisfies 
 \be\la{fu9}\ba &\lim_{\de\rightarrow0}\lim_{\ep\rightarrow0}\int_{0}^{T}\int_{\om} \psi_{\tau} \phi_{\de}[u]^{\ep}\left(\xi_{0}\div (u \otimes u)^{\ep}+\sum_{i=1}^{k}\xi_{i} \div (\ti {u} \otimes \ti{u})_{i}^{\ep}\right)=0.\ea\ee
 \end{lemma}
\begin{proof} 
A careful computation shows 
\bnn\ba &\int_{0}^{T}\int_{\om} \psi_{\tau} \phi_{\de}[u]^{\ep}\left(\xi_{0}\div (u \otimes u)^{\ep}+\sum_{i=1}^{k}\xi_{i} \div (\ti {u} \otimes \ti{u})_{i}^{\ep}-\div (u \otimes [u]^{\ep})\right)\\
&\le \int_{\tau}^{T-\tau} \|u\|_{L^{4}}\left( \|\div (u \otimes u)^{\ep}-\div (u \otimes u^{\ep})\|_{L^{\frac{4}{3}}(V_{0})}+\|\div (u \otimes (u^{\ep}- [u]^{\ep}))\|_{L^{\frac{4}{3}}(V_{0})}\right) \\
&\quad+\sum_{i=1}^{k}\int_{\tau}^{T-\tau} \|u\|_{L^{4}} \left(\| \div (\ti {u} \otimes \ti{u})_{i}^{\ep}-\div (u \otimes \ti{u}_{i}^{\ep})\|_{L^{\frac{4}{3}}(V_{i})}+\| \div ( u \otimes (\ti{u}_{i}^{\ep}-[u]^{\ep})\|_{L^{\frac{4}{3}}(V_{i})}\right)\\
&\le C\int_{0}^{T}\|u\|_{L^{4}}^{2}  \|\na u\|_{L^{2}} \\
&\quad + C\int_{\tau}^{T-\tau}\|u\|_{L^{4}} \left(\|\na u\|_{L^{2}}\|u^{\ep}- [u]^{\ep}\|_{L^{4}(V_{0})}+\|u\|_{L^{4}}\|\na (u^{\ep}- [u]^{\ep}) \|_{L^{4}(V_{0})}  \right)\\
&\quad   +C\sum_{i=1}^{k}\int_{\tau}^{T-\tau}\|u\|_{L^{4}} \left(\|\na u\|_{L^{2}}\|\ti{u}_{i}^{\ep}- [u]^{\ep}\|_{L^{\frac{2q}{q-2}}(V_{i})}+\|u\|_{L^{4}}\|\na (\ti{u}_{i}^{\ep}- [u]^{\ep}) \|_{L^{4}(V_{i})}  \right)\\
&\le C\int_{0}^{T}\left(\|\na u\|_{L^{2}}^{2}+ \|u\|_{L^{4}}^{4} \right),\ea\enn
owing to  Lemma \ref{lem} and Corollary \ref{co}.

By virtue of  \eqref{1.12fu} and  \eqref{fu2},  it satisfies   
\be\la{k7}\ba   \|u\|_{L^{4}(0,T;L^{4})}   \le \|u\|_{L^{\infty}(0,T;L^{2})}^{\frac{q-4}{2q-4}}\|u\|_{L^{\frac{4q}{2q-4}}(0,T;L^{q})}^{\frac{q}{2q-4}}\le C, \ea\ee
 provided $q\ge 4,\,\frac{4q}{2q-4}\le p$, which is equivalent to $q\ge 4,\,\frac{1}{q}+\frac{1}{p}\le \frac{1}{2}.$
Thus,   \be\la{k5} \lim_{\ep\rightarrow0}\int_{0}^{T}\int_{\om} \psi_{\tau} \phi_{\de}[u]^{\ep}\left(\xi_{0}\div (u \otimes u)^{\ep}+\sum_{i=1}^{k}\xi_{i} \div (\ti {u} \otimes \ti{u})_{i}^{\ep}-\div (u \otimes [u]^{\ep})\right)=0.\ee  

Next,  notice from  Lemma \ref{lemma}  that 
\be\la{k6}\ba \int_{0}^{T}\int_{\om} & \psi_{\tau} \phi_{\de}[u]^{\ep}\div (u \otimes[ u]^{\ep})
=-\frac{1}{2}\int_{0}^{T}\int_{\om} \psi_{\tau} u\cdot \na \phi_{\de}   |[ u]^{\ep}|^{2}\\
&\le C\|u\cdot \na \phi_{\de}\|_{L^{2}(0,T;L^{2})}\int_{0}^{T}\int_{\{x:\,dist(x,\p\om)<\de\}}  |[ u]^{\ep}|^{4}\rightarrow0\quad {\rm as}\,\,\de\rightarrow0.\ea\ee
Therefore,   \eqref{fu9} follows from \eqref{k5} and \eqref{k6}.
\end{proof}

 \begin{lemma}\la{l4.2} The pressure term in \eqref{2.1fu}  satisfies \be\la{k3}
 \lim_{\de\rightarrow0} \lim_{\ep\rightarrow0} \int_{0}^{T}\int_{\om} \psi_{\tau}\phi_{\de}[u]^{\ep}\left(\xi_{0}\na P^{\ep} +\sum_{i=1}^{k}\xi_{i}\na  \ti{P}_{i}^{\ep} \right)=0.
 \ee
 \end{lemma}
 \begin{proof}
 From \eqref{pressure}, \eqref{fu6}, and \eqref{k7} we have
 \begin{equation}\la{control pressure}
 \|P\|_{L^{2}(\om)}\le C.
 \end{equation}

Next,  we write 
 \be\la{fu8}\ba &\int_{0}^{T}\int_{\om} \psi_{\tau}\phi_{\de}[u]^{\ep}\left(\xi_{0}\na P^{\ep} +\sum_{i=1}^{k}\xi_{i}\na  \ti{P}_{i}^{\ep} \right)\\
 &=\int_{0}^{T}\int_{\om} \psi_{\tau}\phi_{\de}[u]^{\ep}\left(\xi_{0}\na P^{\ep} +\sum_{i=1}^{k}\xi_{i}\na  \ti{P}_{i}^{\ep} -\na [P]^{\ep}\right) \\
 &\quad +\int_{0}^{T}\int_{\om} \psi_{\tau}\phi_{\de}([u]^{\ep}-u)\cdot \na [P]^{\ep} +\int_{0}^{T}\int_{\om} \psi_{\tau}\phi_{\de}u\na [P]^{\ep}\\
 &=: K_{1}+K_{2}-\int_{0}^{T}\int_{\om} \psi_{\tau}u\cdot \na \phi_{\de} [P]^{\ep}.\ea\ee
 
The terms on the right-hand side will be treated as follows:
 
First,  by  \eqref{u2}, \eqref{1.12fu} and \eqref{control pressure},
 \bnn\ba K_{1}&=\int_{0}^{T}\int_{\om} \psi_{\tau}\phi_{\de}[u]^{\ep}\left(\xi_{0}\na( P^{\ep} -P)+\sum_{i=1}^{k}\xi_{i}\na  (\ti{P}_{i}^{\ep} - P)\right)\\
 &= - \int_{0}^{T}\int_{\om} \psi_{\tau}\left(\na\phi_{\de}[u]^{\ep}\xi_{0}+\phi_{\de}\div [u]^{\ep}\xi_{0}+\phi_{\de}[u]^{\ep}\na \xi_{0}\right) ( P^{\ep} -P) \\
 &\quad -\sum_{i=1}^{k}\int_{0}^{T}\int_{\om} \psi_{\tau}\left(\na\phi_{\de}[u]^{\ep}\xi_{i}+\phi_{\de}\div [u]^{\ep}\xi_{i}+\phi_{\de} [u]^{\ep}\na \xi_{i}\right) (\ti{P}_{i}^{\ep} -P) \\
 &\le  C(\de) \left(\|P^{\ep} -[P]^{\ep}\|_{L^{2}_{loc}(0,T;L^{2}(V_{0}))}+\sum_{i=1}^{k}\|\ti{P}_{i}^{\ep} -[P]^{\ep}\|_{L^{2}_{loc}(0,T;L^{2}(V_{i}))}\right)\ea\enn
which  implies $\lim_{\ep\rightarrow0}K_{1}=0$ due to \eqref{k7}  and \eqref{p9}. 

The limit $\lim_{\ep\rightarrow0}K_{2}=0$ can be proved using similar arguments. 

Finally,  by \eqref{1a}, \eqref{1.12fu},  the Hardy inequality, and \eqref{control pressure} it follows that  
\bnn\ba \int_{0}^{T}\int_{\om} \psi_{\tau}u\cdot \na \phi_{\de}   P\le  C\|u\cdot \na \phi_{\de}\|_{L^{2}(0,T;L^{2})}\int_{0}^{T}\int_{\{x:\,dist(x,\p\om)<\de\}} |P|^{2}\rightarrow0\quad {\rm as}\,\,\de\rightarrow0,\ea\enn
proving the lemma.
\end{proof}

 \bigskip
 
\section*{Acknowledgement} 
R. M. Chen would like to thank Theodore Drivas for helpful discussions. The work of R. M. Chen is partially supported by National Science Foundation under Grant DMS-1613375. The work of Z. Liang is partially supported by the fundamental research funds for central universities (JBK 1805001). The work of D. Wang is partially supported by the
National Science Foundation under grants DMS-1312800 and DMS-1613213. The work of R. Xu is partially supported by the National Natural Science Foundation of China (11871017).


\bigskip 
 
\begin {thebibliography} {99}

\bibitem{ADSW}
I. Akramov, T. Debiec, J. Skipper, and E. Wiedemann, Energy conservation for the compressible Euler and Navier-Stokes equations with vacuum, arXiv: 1808.05029.

\bibitem{BGGTW}
C. Bardos, P. Gwiazda, A. \'Swierczewska-Gwiazda, E. Titi, and E. Wiedemann, On the extension of Onsager's conjecture for general conservation laws, arXiv:1806.02483.

\bibitem{bar}
C. Bardos and  E. S.  Titi, Onsager's conjecture for the incompressible Euler equations in bounded domains, Arch. Ration. Mech. Anal.  \textbf{228}(1)  (2018),  197--207.

\bibitem{bar1}  C. Bardos,   E.  S. Titi, and E. Wiedemann,  Onsager's conjecture with physical boundaries and an application to the vanishing viscosity limit,    arXiv:1803.04939.


\bibitem{cl} R. M. Chen, Z. Liang and D. Wang, Vanishing viscosity for homogeneous incompressible 
flows with physical boundaries, in preparation. 


\bibitem{che}
A. Cheskidov, P. Constantin, S. Friedlander, and R. Shvydkoy, Energy conservation and Onsager's conjecture for the Euler equations, Nonlinearity \textbf{21} (2008), 1233--1252.

\bibitem{che2} A. Cheskidov and X. Luo,  On the energy equality for Navier--Stokes equations in weak-in-time Onsager spaces, arXiv: 1802.05785v2.

\bibitem{con}  P. Constantin,  W. E, and E. S. Titi, Onsager's conjecture on the energy conservation for solutions of Euler's equation, Comm. Math. Phys. \textbf{165}(1) (1994),  207--209.

%

\bibitem{di} R. J.  DiPerna and  P.-L. Lions,  Ordinary differential equations, transport theory and Sobolev spaces,   Invent. Math. \textbf{98} (1989), 511--547.

\bibitem{ngu}  T. D.  Drivas and H. Q.  Nguyen,  Onsager's conjecture and anomalous dissipation on domains with boundary,  arXiv: 180305416v1.

\bibitem{de1} T. D.  Drivas and G. Eyink. Cascades and Dissipative Anomalies in Compressible Fluid Turbulence, Phys. Rev. X 8, 011023 (2018).

\bibitem{de2} T. D.  Drivas and G. Eyink. An Onsager Singularity Theorem for Turbulent Solutions of Compressible Euler Equations,  Comm. Math. Phys., \textbf{359} (2018), 733--763.

\bibitem{de} T. D.  Drivas and G. Eyink. An Onsager singularity theorem for Leray solutions of incompressible Navier--Stokes, arXiv:1710.05205.

\bibitem{du} J. Duchon and  R. Robert, Inertial energy dissipation for weak solutions of incompressible Euler and Navier--Stokes equations, Nonlinearity  \textbf{13} (2000), 249--255.

\bibitem{esc} L. Escauriaza and S. Montaner, Some remarks on the $L^p$ regularity of second derivatives of solutions to non-divergence elliptic equations and the Dini condition, Rend. Lincei. Mat. Appl., \textbf{28} (2017), 49--63.

\bibitem{evans} L. C.   Evans, Partial differential equations, Second edition, Graduate Studies in Mathematics, \textbf{19}. American Mathematical Society, Providence, RI, 2010.

\bibitem{ey} G.  Eyink. Energy dissipation without viscosity in ideal hydrodynamics: I. Fourier analysis and local energy transfer, Phys. D \textbf{78} (1994), 222--240.

\bibitem{Feireisl1} E. Feireisl, {Dynamics of Viscous Compressible Fluids}  Oxford Lecture Series in Mathematics and its Applications, 26. Oxford Science Publications. The Clarendon Press, Oxford University Press, New York, 2004.

\bibitem{fei} E. Feireisl, A. Novotn$\acute{y}$, and H. Petzeltov$\acute{a}$, On the existence of globally defined weak solutions to the Navier--Stokes equations,  J. Math. Fluid Mech.  \textbf{3} (2001),  358--392.

\bibitem{fei1}  E. Feireisl, P. Gwiazda,  A. Swierczewska-Gwiazda, and  E. Wiedemann, Regularity and energy conservation for the compressible Euler equations,  Arch. Ration. Mech. Anal. \textbf{223} (2017), 1375--1395.
 
%


\bibitem{ka} Y. Kaneda, T. Ishihara, M. Yokokawa, K. Itakura, and A. Uno. Energy dissipation rate and energy spectrum in high resolution direct numerical simulations of turbulence in a periodic box, Phys. Fluids, \textbf{15} (2003), L21--L24.

\bibitem{ku} A. Kufner,  O.  John,  and S.  Fu$\breve{c}\acute{i}$k,    Function Spaces, Academia, Prague  (1977).

\bibitem{lady} O. A. Lady$\breve{\textrm{z}}$enskaja, V. A. Solonnikov and N. N. Ural'ceva. Linear and quasilinear equations of parabolic type, translated from the Russian by S. Smith, Translations of
Mathematical Monographs, Vol. 23, American Mathematical Society, Providence, RI, 1968.

\bibitem{leray}  J. Leray,  Sur le mouvement dun liquide visqueux emplissant lespace,   Acta Math.   \textbf{63}1 (1934), 193--248.

\bibitem{shv}   T. M.   Leslie and  R. Shvydkoy, The energy balance relation for weak solutions of the density-dependent  Navier--Stokes equations,  J. Differ. Equ. \textbf{261}  (2016), 3719--3733.

\bibitem{shv2} T. M.   Leslie and  R. Shvydkoy, Conditions implying energy equality for weak solutions of the Navier--Stokes equations, SIAM J. Math. Anal., \textbf{50} (2018), 870--890.

\bibitem{lions} J. L. Lions, Sur la r\'egularit\'e et l`unicit\'e des solutions turbulentes des \'equations de Navier Stokes, Rend. Sem. Mat. Univ. Padova, \textbf{30} (1960), 16--23.

 \bibitem{lion1}  P.-L.  Lions, Mathematical topics in fluid mechanics, Vol. \textbf{1}. Incompressible models. Oxford Lecture
Series in Mathematics and its Applications, 3. Oxford Science Publications. The Clarendon Press,
Oxford University Press, New York, 1996.
 
 \bibitem{lion2} P.-L. Lions,  Mathematical topics in fluid mechanics, Vol. \textbf{2}. Compressible models. Oxford Lecture
Series in Mathematics and its Applications, 10. Oxford Science Publications. The Clarendon Press,
Oxford University Press, New York, 1998. 

\bibitem{mauro} J. A. Mauro, On the regularity properties of the pressure field associated to a Hopf weak solution to the Navier--Stokes equations, Pliska Stud. Math. Bulgar. \textbf{23} (2014), 95--118.

\bibitem{nezza} E. Di Nezza, G. Palatucci and E. Valdinoci, Hitchhiker's guide to the fractional Sobolev spaces, Bull. Sci. Math., \textbf{136} (2012), 521--573.
 
\bibitem{on} L. Onsager, Statistical Hydrodynamics, Nuovo Cimento (Supplemento)  \textbf{6} (1949), 279--287. 
 
 \bibitem{pe} B. Pearson, P. -$\mathring{\textrm{A}}$. Krogstad, and W. Van De Water. Measurements of the turbulent energy dissipation rate, Phys. Fluids, \textbf{14} (2002), 1288--1290.
 
 \bibitem{se} J. Serrin, The initial value problem for the Navier--Stokes equations, 1963 Nonlinear Problems (Proc.Sympos., Madison, WI., 1962) pp. 69--98   Univ. of Wisconsin Press, Madison, WI.
 
\bibitem{sr} K. R. Sreenivasan. On the scaling of the turbulence energy dissipation rate, Phys. Fluids, \textbf{27} (1984), 1048--1051.
 
\bibitem{shi}  M. Shinbrot, The energy equation for the Navier--Stokes system, SIAM J. Math. Anal. \textbf{5} (1974), 948--954.
 
\bibitem{sim} J. Simon,  Compact sets in the space $L^{p}(0,T; B)$,   Ann. Mat. Pura  Appl. \textbf{146} (1987), 65--96.
 
\bibitem{yu1}  C.  Yu,    Energy conservation for the weak solutions of the compressible Navier--Stokes equations,    Arch. Ration. Mech. Anal. \textbf{225}(3) (2017),   1073--1087.

\bibitem{yu3}  C.  Yu,   A  new proof to the energy conservation for the Navier--Stokes equations,  arXiv: 1604.05697v1.  
 
\bibitem{yu2}  C.  Yu,   The energy conservation for the   Navier--Stokes equations in bounded domains, arXiv: 1802.07661.

 \end {thebibliography}
\end{document}